\documentclass{article}

\usepackage{latexsym,amsmath,amsfonts,amscd}
\usepackage{mathrsfs} 
\usepackage{amsthm}
\usepackage{amssymb}
\usepackage{graphicx}
\usepackage{epstopdf}
\usepackage{ifthen}
\usepackage{etoolbox}
\usepackage{color}
\usepackage{cancel}

\usepackage[usenames,dvipsnames]{xcolor}
\usepackage{soul}
\usepackage[normalem]{ulem}

\usepackage{titling}

\usepackage{stmaryrd} 

\usepackage{nameref}
\usepackage{zref-xr}
\zxrsetup{toltxlabel}
\zexternaldocument*[BilayerPearl-]{../../BilayerPearling/PearlingBilayerV35/PearlingforBilayerV35}

\usepackage{hyperref}

\numberwithin{equation}{section}

\newtheorem{theo}{Theorem}[section]
\newtheorem{assum}[theo]{Assumption}
\newtheorem{lemma}[theo]{Lemma}

\newtheorem{prop}[theo]{Proposition}
\newtheorem{corol}[theo]{Corollary}
\newtheorem{definition}[theo]{Definition}

\usepackage{wrapfig}
\newtoggle{Draft}
\toggletrue{Draft}

\usepackage{cancel}

\date{\today}

\def\beal{\begin{align*}}
\def\eeal{\end{align*}}

\def\beq{\begin{equation}}
\def\eeq{\end{equation}}
\def\beqa{\begin{eqnarray}}
\def\eeqa{\end{eqnarray}}
\def\bpm{\begin{pmatrix}}
\def\epm{\end{pmatrix}}
\def\cC{{\cal C}}

\def\cE{{\cal E}}
\def\cF{{\cal F}}
\def\cG{{\cal G}}

\def\cL{{\cal L}}

\def\cZ{{\cal Z}}
\def\eps{\varepsilon}

\def\tJ{{\tilde{J}}}

\def\hpsi{\psi}
\def\spsi{\tilde{\psi}}
\def\ov{\overline{v}}
\def\hlam{\hat{\lambda}}

\def\oQ{\overline{Q}}
\def\valpha{\vec{\alpha}}
\def\vk{\vec{k}}

\def\vv{\vec{v}}
\def\vw{\vec{w}}

\def\bN{{\bf N}}
\def\bT{{\bf T}}

\def\mbbR{\mathbb{R}}
\def\RR{\mathbb{R}}

\def\LL{\mathbb{L}}

\def\NN{\mathbb{N}}
\def\fil{{{\it f}}}

\newcommand{\Frac}[2]{\frac{\textstyle #1}
                           {\textstyle #2}}

\newcommand{\leavethisout}[1] {}

\newcommand{\norm}[1]{\left\lVert#1\right\rVert}

\newcommand{\draft}[1]{}

\hypersetup{
	colorlinks=true,
	linkcolor=WildStrawberry,
	citecolor=ProcessBlue,
}

\numberwithin{equation}{section}

\newcommand\bt[1]{\begin{tabular}{#1}}
\newcommand\et{\end{tabular}}

\def\bc{\begin{center}}
\def\ec{\end{center}}

\newtheorem{Theorem}{Theorem}

\newtheorem{Def}[Theorem]{Definition}

\def\vkappa{\vec{\kappa}}
\def\vz{\vec{z}}
\def\Hbl{H}

\def\bL{\mathbb{L}}

\def\RR{\mathbb{R}}
\def\LL{\mathbb{L}}

\usepackage{lineno}

\begin{document}

\title{Pearling Bifurcations in the strong Functionalized Cahn-Hilliard Free Energy}

\author{Noa Kraitzman\thanks{Department of Mathematics, University of Utah, noa@math.utah.edu} \,\&\, Keith Promislow\thanks{Department of Mathematics, Michigan State University, kpromisl@math.msu.edu}}

\maketitle
\begin{abstract}
The Functionalized Cahn-Hilliard free energy supports phase separated morphologies of distinct codimension, including codimension-one
bilayer and codimension-two filament morphologies. We characterize the linear stability of bilayer and filament morphologies associated to
hypersurfaces within the strong functionalization scaling. In particular we show that the onset of the pearling instability, which
triggers fast in-plane oscillations associated to bifurcation to higher codimensional morphology, is controlled by the functionalization parameters
and the spatially constant value of the far-field chemical potential.  Crucially, we show that onset of pearling is independent of the shape
of the defining  hypersurface.
\end{abstract}

\section{Introduction}\label{Sec:FCHFE}

Models of amphiphilic materials date to the empirical analysis  of Teubner and Strey \cite{teubner1987origin} and Gompper and Schick \cite{gompper1990correlation} who
studied the small-angle x-ray scattering (SAXS) data of microemulsions of oil, water, and surfactant. For regimes with a predominance of water and a minority phase
comprised of surfactant and oil, they arrived at a free energy landscape
for the surfactant-oil volume fraction $u\in H^2(\Omega)$
\beq\label{GS-energy}
\cF_{\rm GS}(u) = \int_\Omega \frac{\eps^4}{2} |\Delta u|^2 - \eps^2 G_1(u)\Delta u + G_2(u)\, dx,
\eeq
where the function $G_1$ encodes the material amphiphilicity, taking distinct signs in the surfactant-oil phase $u\approx 1$ and in the water phase $u\approx 0$,
and $G_2$ describes the free energy of spatially homogeneous blends.  The parameter $\eps\ll1$ scales homogeneously with space and
denotes the ratio of the length of the amphiphilic (surfactant) molecule to the domain size $\Omega\subset{\RR}^d$.
Completing the square in $\Delta u$ leads to an equivalent formulation
\beq
\label{GS-energy-sq}
\cF_{\rm GS}(u) = \int_\Omega \frac12 \left( \eps^2\Delta u -G_1(u)\right)^2  +P(u)\, dx,
\eeq
where the residual term $P(u) := G_2(u)-\frac12 G_1^2(u)$. The Functionalized Cahn-Hilliard
(FCH) free energy corresponds to the special case in which the residual is asymptotically small with respect to $\eps$
and the term  $G_1(u)=W'(u)$ where $W$  is a double-well potential with minima at $u=b_\pm$
whose unequal depths are normalized so that $W(b_-)=0>W(b_+)$ and is non-degenerate in the sense that $\alpha_\pm:=W''(b_\pm)>0$.
Here $u=b_-$ is associated to a bulk solvent phase, while the quantity $u-b_->0$ is proportional to the density of the amphiphilic phase. More specifically
we consider the distinguished limit in which the residual term scales as $\eps^p$, and the resulting form is a linear combination of the ``quadratic'' and the ``functionalization'' terms,
\beq\label{e:FCHFE}
 \cF(u) := \int\limits_\Omega \frac {1}{2} \left(\eps^2\Delta u-W^\prime(u) \right)^2-
 \eps^p \left(\frac{\eps^2\eta_1}{2}|\nabla u|^2+\eta_2W(u)\right)\, dx.
 \eeq
 The dominant quadratic term corresponds to the square of the variational derivative of a Cahn-Hilliard free energy, \cite{Cahn1958free},
 \beq\label{e:CHFE}
 \cE(u) = \int_\Omega \frac{\eps^2}{2}|\nabla u|^2 + W(u)\, dx,
 \eeq
 whose \emph{minimizers} over $H^2(\Omega)$, subject to prescribed volume fraction, are related to minimal surfaces.
 The zeros of the quadratic term within the FCH are precisely the \emph{critical points}, in particular the saddle points,
 of the associated Cahn-Hilliard free energy $\cE$: the minimization of the FCH free energy
 is achieved by searching for the critical points of the Cahn-Hilliard free energy $\cE$ which minimize the functionalization terms of $\cF$.
 The functionalization terms, parameterized by $\eta_1>0$ and $\eta_2\in\RR$  are analogous to the surface and
volume energies typical of models of charged solutes in confined domains, see \cite{scherlis2006unified} and particularly equation (67) of \cite{andreussi2012revised}.
The minus sign in front of $\eta_1$ is of considerable significance --  it incorporates the propensity of the
amphiphilic surfactant phase to drive the creation of interface. Indeed, experimental tuning of solvent quality
shows that  morphological instability in amphiphilic mixtures is associated to (small) negative values of
surface tension, \cite{zhu2009tuning} and \cite{zhu2012interfacial}. There are two natural choices for the exponent $p$ in the functionalization terms.
In the strong functionalization, $p=1$, the functionalization terms dominate the Willmore corrections from the squared variational term.
The weak functionalization, corresponding to $p=2$, is the natural scaling for the $\Gamma$-limit as the curvature-type
Willmore terms appear at the same asymptotic order as the functional terms. In this paper we focus on the strong scaling of the FCH free energy.

For a cubical domain $\Omega=[0,L]^d\subset{\RR}^d$ subject to periodic boundary conditions, the first variation of~$\cF$ at $u\in H^4(\Omega)$,
is denoted by the chemical potential $\mu$ and takes the form
\begin{align}\label{e:mu}
	\mu:=\frac{\delta \cF}{\delta u}(u) &= (\eps^2\Delta-W''(u)+\eps\eta_1)(\eps^2\Delta u-W'(u))+\eps\eta_dW'(u),
\end{align}
where~$\eta_d:=\eta_1-\eta_2$.  The Functionalized Cahn-Hilliard equation arises as the $H^{-1}$ gradient flow of the FCH free energy
\beq
\label{e:FCH-eq}
u_t = \Delta \mu(u),
\eeq
and is typically considered in conjunction with periodic boundary conditions on the cubical domain $\Omega$.
This work is the first of a two-part effort that addresses the slow, curvature-driven dynamics and the linear stability of families
of codimension one bilayer and codimension two filament morphologies.  In this paper we construct families of admissible bilayer and filament
morphologies and address their linear stability with respect to the pearling bifurcation within the context of the strong FCH gradient flow,
(\ref{e:FCH-eq}).  Admissibility, defined rigorously in Definition\,\ref{def:admissible} and \ref{def:admissible-2}, most significantly requires that
the morphology is sufficiently far from self-intersection. In the companion paper, \cite{Christlieb2017Competition} we present a multiscale analysis which shows that well separated
filament and bilayer morphologies evolve according to a competitive quenched mean-curvature driven flow mediated through the common value
of the spatially constant far-field chemical potential, $\mu$.  The far-field value of the chemical potential serves as a bifurcation parameter, potentially triggering two classes of instabilities for each codimension of morphology.
Indeed, in \cite{HAYRAPETYAN2014SPECTRA} it was shown rigorously for the weak FCH that the pearling instabilty and geometric meander, also called the fingering instability, are the only possible instabilities for the
bilayer morphology. The pearling instability is the focus of this paper, when triggered it is manifest on a fast $O(\eps)$ time-scale and
leads to a periodic modulation of the width of the corresponding bilayer or filament morphology.  Conversely, the geometric meander drives the
shape of the underlying morphology on the slow $O(\eps^{-1})$ time-scale, and leads either to motion by curvature or motion against curvature, depending upon the value of the far-field chemical potential, $\mu$,
measured against a morphology specific reference value.  Motion by curvature decreases surface area, while motion against curvature
induces buckling or meander type evolution that if unchecked generically leads to self-intersection of the underlying morphology and the
generation of finger like protrusions. Both the pearling instability and the onset of motion against curvature are typically associated to
the transition of a morphology to a higher codimention, although pearled morphologies have been shown to exist as equilibrium when the
underlying interface is circular or flat, \cite{PW-15}. Significantly, the geometric evolution couples to the pearling
bifurcation as the growth of interface leads to a change in the far-field value of $\mu$ on the slow time scale, which we show can trigger the pearling bifurcation, see Figure\,\ref{f:Manpearl} for an example of motion against curvature triggering a pearling bifurcation in a simulation of the FCH equation (\ref{e:FCH-eq}).

\begin{figure}[h!]
\begin{center}
\begin{tabular}{ccc}
    \includegraphics[width=2.0in]{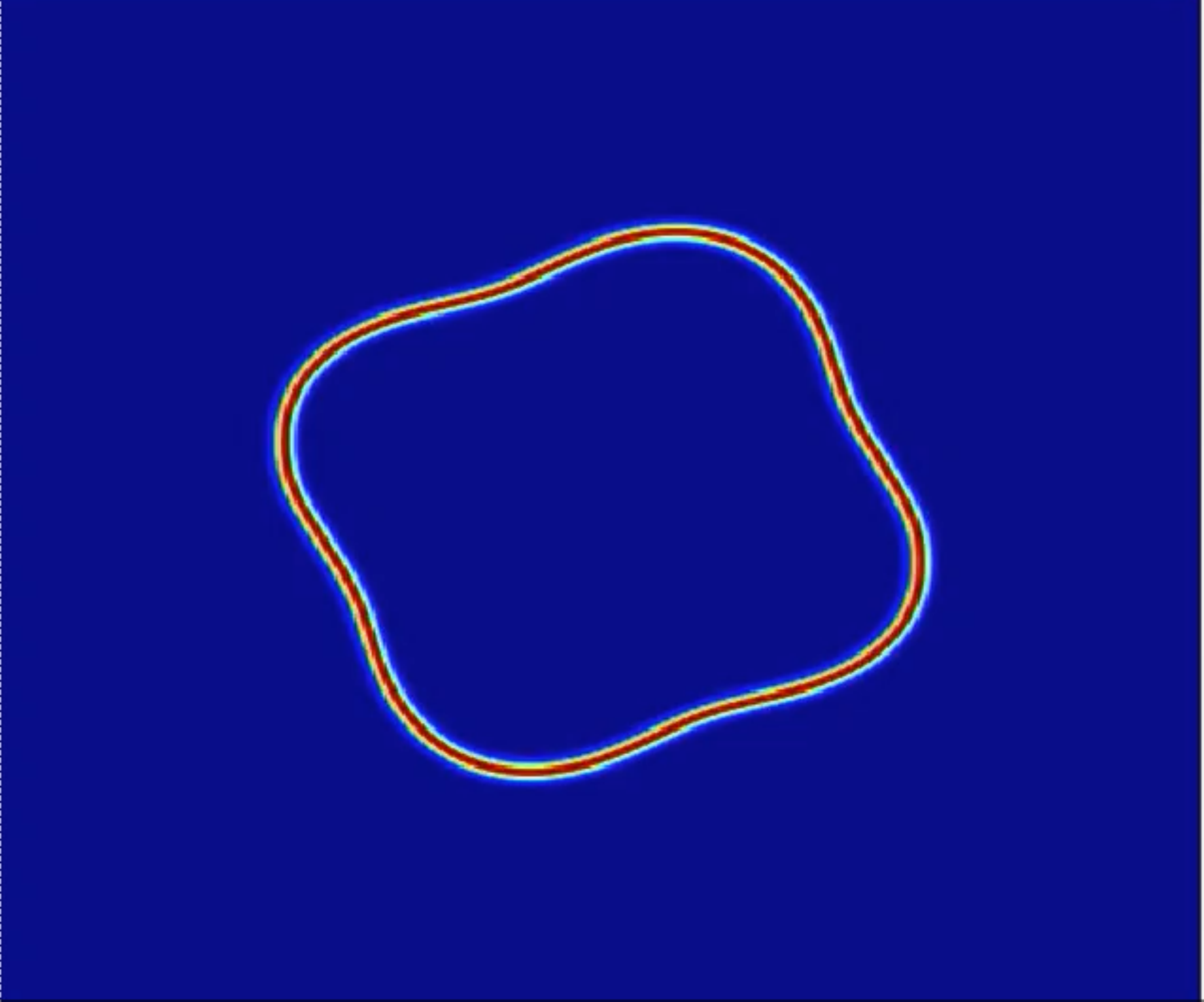}&
    \includegraphics[width=2.0in]{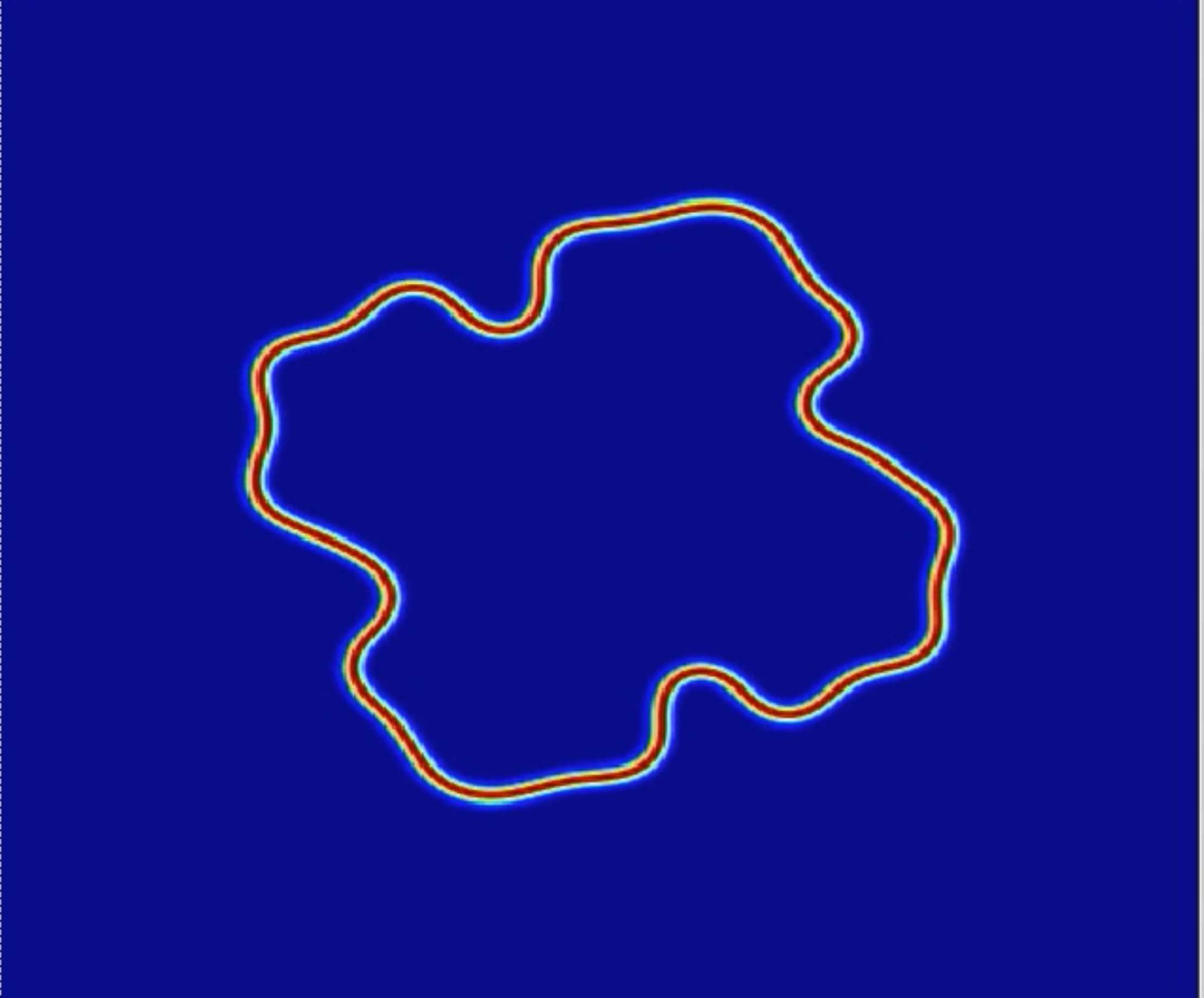}&
      \includegraphics[width=2.0in]{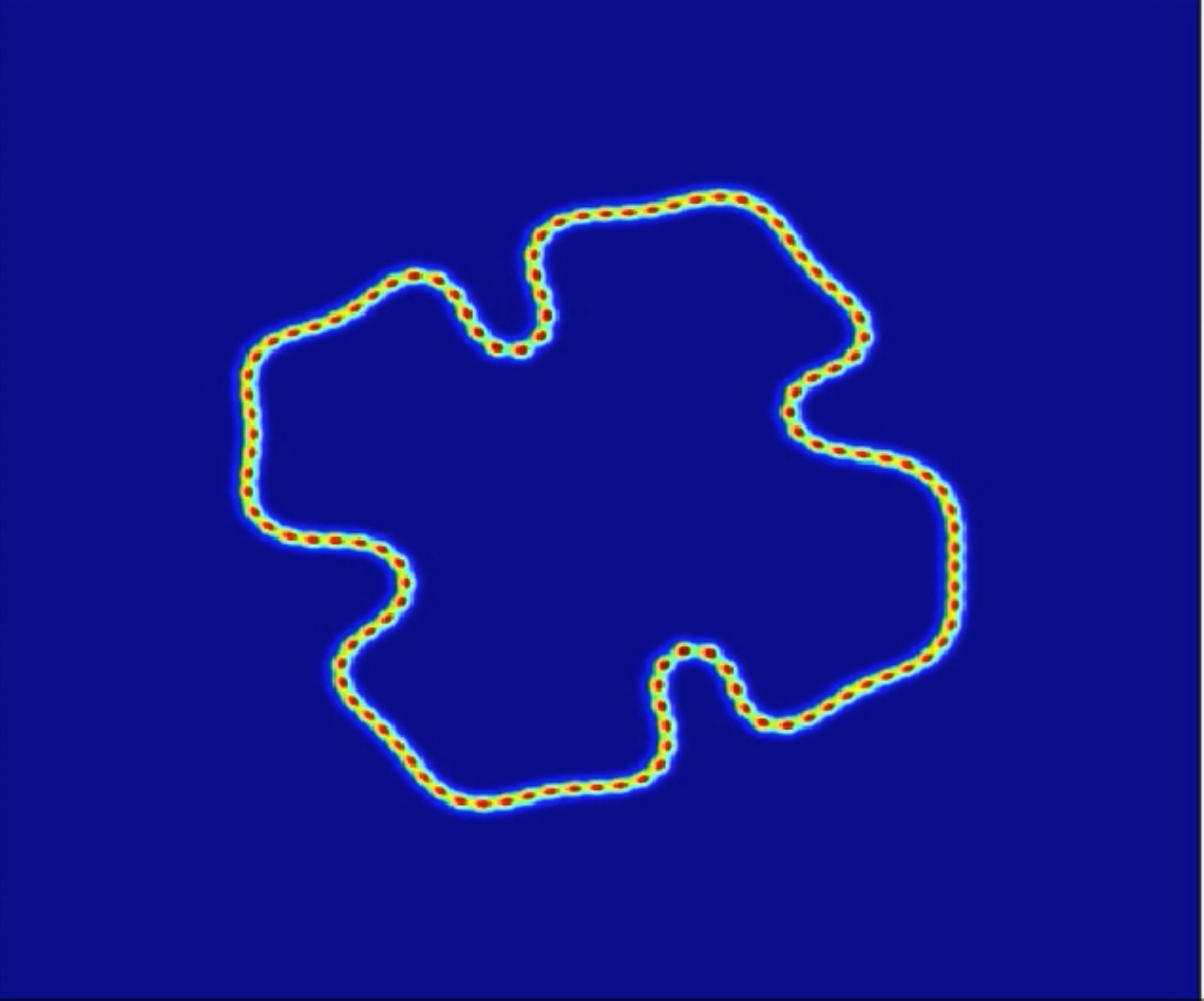}
  \end{tabular}
  \end{center}
\caption{Simulations of of the strong FCH gradient flow, (\ref{e:FCH-eq}) on a 2D domain. The initial data (left) is a codimension one bilayer morphology with an elevated value of the far field chemical potential corresponding to motion against curvature. The result is a meandering evolution that lengthens the curve and reduces the far field chemical potential (center), eventually triggering a pearling instabilty, (right).}
 \label{f:Manpearl}
\end{figure}


In this paper we investigate the spectrum of the linearization of the right-hand side of the FCH equation about bilayer and filament morphologies
constructed from admissible hypersurfaces, and determine the dependence of the spectrum upon the functionalization parameters and the
spatially constant value of the far-field chemical potential.  Within the context of our construction the far-field chemical potential is proportional to the
far-field density of free amphiphilic material, $u-b_-$, and serves as the long-range communication mechanism between the spatially localized
bilayer and filament morphologies. The linearization of the FCH equation about a profile $u\in H^4(\Omega)$ takes the form $\Delta \LL$, where the second variational derivative of the FCH free energy takes the form
\beq
\label{e:LL}
	\LL:=\frac{\delta^2\cF}{\delta u^2}(u) = \left(\eps^2\Delta-W''(u)+\eps\eta_1\right)\left(\eps^2\Delta-W''(u)\right)-\left(\eps^2\Delta u-W'(u)\right)W'''(u)+\eps\eta_dW''(u).
\eeq
The spectrum of both $\LL$ and $\Delta \LL$ are purely real, and small eigenvalues of $\Delta \LL$ have an
explicit mapping onto those of $\LL$, see \cite{Doelman2014meander}, so that the bifurcation of unstable modes in $\Delta \LL$ is controlled by
the zero crossing of eigenvalues in $\LL$. The central result of this paper is a rigorous analysis of the eigenvalue problem
 \beq
    \LL\Psi = \Lambda\Psi,
 \eeq
 about bilayer and filament morphologies in the context of the strong FCH free energy. In particular, we establish explicit pearling stability conditions for both bilayers in $\RR^d$ and filaments in $\RR^3$ that apply uniformly to  all admissible morphologies. This extends the results of \cite{Doelman2014meander}, which considered the case of constant single-curvature bilayer morphologies that are exact equilibria of the FCH system. The key obstacle to this extension is to retain estimates that are uniformly valid over the of pearling eigenvalues that are both asymptotically large in number and asymptotically close together as they interact through the nonconstant interfacial curvatures.

 Each bilayer morphology is associated to a codimension one interface $\Gamma_b$ embedded in $\Omega$.  Within the reach of
 the interface, the $\eps$-scaled signed-distance $z$ to the interface is well defined and smooth, and the leading order bilayer profile is defined as the solution $\phi_b$ of the second order equation
\beq
\label{e:BLcp}
 \partial_z^2 \phi_b = W'(\phi_b),
 \eeq
 which is homoclinic to the left well $b_-$ of $W$. The leading-order bilayer morphology $U_b=U_b(\cdot;\Gamma_b)\in H^4(\Omega)$
 associated to $\Gamma_b$, also called the \emph{dressing} of $\Gamma_b$ with $\phi_b$, takes the values $U_b(x)=\phi_b(z(x))$
 within the reach and is extended smoothly to the constant value $b_-$ on $\Omega_\pm$, see Figure\,\ref{f:SpectrumBilayer} (left) for a graphical depiction and Section\,\ref{sec:coordinates} for a precise definition. The second variation of the FCH, given in (\ref{e:LL}), evaluated at $u=U_b$, is denoted $\LL_b$. It inherits considerable structure from the form of the FCH free energy; indeed it is a perturbation of the square of self-adjoint operator. To illuminate this structure we introduce the linearization of (\ref{e:BLcp}) about $\phi_b$,
 \beq
 \label{e:L0-def}
 L_{b,0}:=\partial_z^2 - W''(\phi_b).
 \eeq
When acting on functions with support within the reach of $\Gamma_b$, it is instructive to think of the operator $\LL_b$ in the form
\beq
\LL_b =  (L_{b,0}+\eps^2 \Delta_s)^2 + O(\eps),
\eeq
where $\Delta_s$ is the Laplace-Beltrami operator associated to $\Gamma_b$. The operator $\LL_b$ is positive except
where the positive eigenspaces of $L_{b,0}$ balance against the negative Laplace-Beltrami operator, in which case the $O(\eps)$
perturbations become relevant. Viewed as on operator on $L^2(\RR)$, $L_{b,0}$ is Sturmian with a positive ground state eigenvalue $
\lambda_{b,0}>0$ and eigenfunction $\hpsi_{b,0}>0$ and a translational eigenmode $\lambda_{b,1}=0$ associated to the eigenfunction $
\hpsi_{b,1}$ which is a rescaling of the translational mode $\phi_b^\prime.$  In \cite{HAYRAPETYAN2014SPECTRA} it was shown that for each
set of admissible codimension-one interfaces there exists $\sigma_K>0$, which may be chosen independent of $\eps>0$  such that the set of
eigenmodes associated to $\bL_b$ with eigenvalue $\lambda<\sigma_K$ is finite dimensional and is composed of two subsets, the
\emph{pearling eigenmodes} $\{\Psi_{b;0,n}\}_{n=N_1}^{N_2}$ and the \emph{meander eigenmodes} $\{\Psi_{b;1,n}\}_{n=0}^{N_3}$.
Moreover these eigenmodes are localized within the reach of $\Gamma_b$ and at leading order are a tensor product of an eigenmode of
$L_{b,0}$ times an eigenmode $\Theta_n$ associated to the Laplace-Beltrami operator, $-\Delta_s$ of interface $\Gamma_b$.
 \begin{figure}[h!]
\begin{center}
\begin{tabular}{cc}
    \includegraphics[width=2.7in]{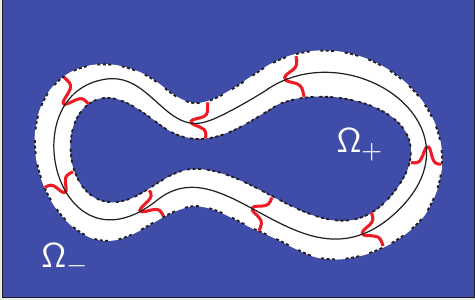} &
    \includegraphics[width=3.3in]{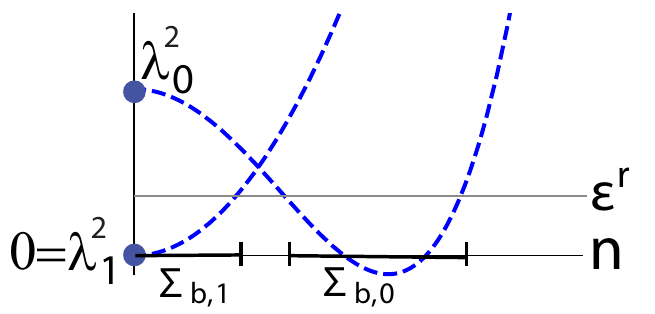}
\end{tabular}
  \end{center}
\caption{(Left) An admissible codimension one interface $\Gamma_b$, denoted by the thin black line, with its reach denoted in white. The remainder of $\Omega$ is comprised of the domains $\Omega_\pm$ exterior to the reach. The dressing of $\Gamma_b$ with the bilayer profile $\phi_b$, is sketched via the red curves within the reach. (Right)  A depiction of the pearling, $\Sigma_{b,0}(r)$, and meander, $\Sigma_{b,1}(r)$ index sets associated to the spectrum of $\LL_b$. The vertical and horizontal axis denote the eigenvalue (real) and the Laplace-Beltrami
wave-number $n$. These center-unstable spectra control the geometric dynamics and triggers the pearling instabilities of bilayer interfaces.}
 \label{f:SpectrumBilayer}
\end{figure}
That is the eigenmodes take the form
\beq
\label{e:BL-tensor}
     \Psi_{b; j,n} = \hpsi_{b,j}(z)\Theta_n(s) +O(\eps),
 \eeq
for $j=0, 1$ and $n$ ranging over the index set $\Sigma_{b,j}(r)$ defined via the Laplace-Beltrami eigenvalues $\beta_n\geq 0$
associated to $-\Delta_s$,
\beq
\label{e:SigmaBl}
\Sigma_{b,j}(r):= \left\{n\in\NN_+\,\bigl|~(\lambda_{b,j}-\eps^2\beta_n)^2\leq \eps^r \right\},\qquad j=0,1,
\eeq
for some fixed $r\in(0,1)$. The value of $r$ controls both the size of the index sets and the coercivity of $\LL_b$ on the orthogonal
complement of the associated linear spaces; the determination of its optimal value is the culmination of the technical analysis of this paper.
The pearling spectrum, associated to $\Sigma_{b,0}(r)$, arises through the balance of the  ground-state eigenvalue $\lambda_{b,0}>0$
of $L_{b,0}$ against appropriately large Laplace-Betrami wave number eigenspaces of the Laplace-Beltrami operator, and induces the
oscillations of the bilayer width seen in Figure\,\ref{f:Manpearl} (right).  The meander spectrum, associated to $\Sigma_{b,1}(r)$, corresponds to
lower Laplace-Beltrami wave numbers that balance the translational eigenvalue $\lambda_{b,1}=0$, and governs the slow geometric
motion of the underlying interface, evidenced in the transition from Figure\,\ref{f:Manpearl} (left) to (center). Weyl's asymptotics for the Laplace-Beltrami eigenvalues  establish that $\beta_n\sim n^{\frac{2}{d-1}},$ so that $N_{b,j}:=|\Sigma_{b,j}(r)|$ grows as a negative power of $\eps>0$. Indeed, there exist  constants $a_0,b_0\in\RR_+$ independent of $\eps$ such that $n\in\Sigma_{b,0}(r)$ satisfy
\beq
\eps^{1-d}(a_0-b_0\eps^{r(d-1)})\leq n \leq \eps^{1-d}(a_0+b_0\eps^{r(d-1)}),
\label{Nj0-asymp}
\eeq
while $n\in\Sigma_{b,1}(r)$ satisfy $0<n<\eps^{(1-d)(1-\frac{r}{4})}b_1,$ for some $b_1>0.$ In particular, $\Sigma_{b,0}(r)$ and $\Sigma_{b,1}(r)$
are disjoint for $\eps$ sufficiently small and $r>0$, see Figure\,\ref{f:SpectrumBilayer} (right).

 In \cite{Doelman2014meander} bilayer equilibria of the strong FCH where explicitly constructed as dressings of single-curvature codimension-one interfaces in $\RR^d$,
 corresponding to cylinders and spheres in $\RR^3$. The onset of both the pearling and the meander instability where characterized in terms of the functionalization parameters,  with the analysis of the meander instability hinging on a very degenerate, higher-order expansion. In \cite{PW-15},  it was shown that the onset of the pearling instability for constant curvature interfaces in $\RR^2$ is associated to the creation of a pearled bilayer equilibrium, modulo a non-degeneracy condition.
 Conversely, for single curvature bilayer equilibria in $\RR^2$, that is the dressings of radially symmetric interfaces, it was shown in \cite{HAYRAPETYAN2016Nonlinear} that for the weak FCH the assumption of pearling stability implies the full nonlinear stability of the underlying radially symmetric bilayer, and that the relaxation to the underlying equilibrium circle is along the family of bilayer morphologies constructed herein following a geometric flow corresponding to a linearized version of the mass-preserving Willmore flow derived formally in \cite{dai2013geometric}.

We construct bilayer morphologies obtained as dressings of admissible interfaces for which the far-field
chemical potential is specified. These are approximate solutions of
\beq\label{e:FCH-static}
\mu(u)=\eps\hlam,
\eeq
where the scaled $\hlam\in\RR$ is simultaneously the Lagrange multiplier associated to prescribed total mass and the kernel of the Laplacian prefactor in (\ref{e:FCH-eq}).   The approximate solutions incorporate corrections to the bilayer morphology and take the form
\beq \label{e:ub-exp}
 u_b (x;\Gamma_b,\hlam) = U_b(x)+ \eps U_{b,1}(x),
 \eeq
where $U_{b,1}=U_{b,1}(\cdot;\Gamma_b,\hlam)$ includes a far-field modification to the value of $u_b$
\beq \label{e:ub-ff}
u_{b}(x) = b_-+ \eps \Frac{\hlam}{\alpha_-^2},
\eeq
outside of the reach of $\Gamma_b$. These bilayer morphologies are quasi-equilibria of (\ref{e:FCH-eq}) and include the equilibrium
solutions considered in \cite{Doelman2014meander}, however approximation of the asymptotically large sets of pearling and meander
eigenvalues is complicated by the coupling induced between the eigenvalues due to the non-constant curvatures of the admissible interfaces.
This result, established in section \ref{Sec:Pearling-Bl}, is summarized below and given in more precision in Theorem\,\ref{PB-PearlingCond}.

\begin{theo}\label{Thm:Main-bl}
Fix the functionalization parameters $\eta_1>0$ and $\eta_2\in\RR$, and double well $W$ with the associated constant $S_b$ defined in (\ref{e:S-bl}).
Then for any family of admissible codimension-one interfaces,  see Definition \ref{def:admissible}, there exists $\eps_0>0$ such that
for all $\eps\in(0,\eps_0)$ the bilayer morphology, $u_b(\cdot;\Gamma_b, \hlam),$ constructed  in~(\ref{e:ub-exp}) by the dressing of an admissible codimension-one interface, $\Gamma_b$,
is spectrally stable with respect to pearling if and only if the scaled Lagrange multiplier~$\hlam$ introduced in (\ref{e:FCH-static}),
satisfies the bilayer pearling stability condition
\beq\label{e:pearl-cond-bl}
 \hlam S_b + \lambda_{b,0}(\eta_1-\eta_2)\|\hpsi_{b,0}\|_{L^2(\RR)}^2<0,
\eeq
where $\psi_{b,0}$ is the ground state eigenfunction of the operator $L_{b,0}$ introduced in (\ref{e:L0-def}) with eigenvalue $\lambda_{b,0}>0.$
The quantity $\hlam$ is related to the far-field value of the bilayer morphology through (\ref{e:ub-ff})
\end{theo}
In the analysis of \cite{Doelman2014meander}, a family
of double-wells was considered for which the constant $S_b$ was uniformly positive; numerical evaluation of $S_b$ shows that it can be of
either sign, and the switching of the sign of $S_b$ has a significant impact on association of pearling stability regimes with the far-field chemical potential, and is addressed in detail in the companion paper, \cite{Christlieb2017Competition}.

We also construct filament morphologies from dressings of codimension-two hypersurfaces, and derive associated
pearling stability conditions. This however requires two assumptions. The first, Assumption\,\ref{CS-Assum:Kernel}, specifies the size of the
positive eigenspace and the kernel of the linearization restricted to two subspaces. For the codimension-one case, this assumption is naturally
satisfied through the Sturm-Liouville theory. The second assumption, Assumption\,\ref{CS-Assum:coercive},  asserts the uniform coercivity
of the linearization on the perpendicular of the tensor product space, $X_\fil$, defined in (\ref{e:Xfil}), associated to the codimension-two pearling
and meander eigenvalues. This result was rigorously established for the codimension-one case in \cite{HAYRAPETYAN2014SPECTRA}, but is
outside the scope of this paper. Modulo these assumptions, we establish in section \ref{Sec:Pearling-Pr} that the sign of the pearling and
meander eigenvalues can be characterized for an entire class of admissible codimension-two filaments in terms of the functionalization parameters and the far-field chemical potential. This result is summarized below and given in full precision in Theorem\,\ref{PP-PearlingCond}.

\begin{theo}\label{Thm:Main-pr}
Fix the functionalization parameters $\eta_1>0$ and $\eta_2\in\RR$, and double well $W$ with associated constant $S_\fil$ defined in
(\ref{e:S-pr}), for which Assumptions \ref{CS-Assum:Kernel} and \ref{CS-Assum:coercive} hold. Then for any family of admissible codimension-two filaments, see Definition \ref{def:admissible-2}, there exists $\eps_0>0$ such that for all $\eps\in(0,\eps_0)$ the filament morphology
constructed  in~(\ref{e:PR}) by the dressing of an admissible codimension-two hypersurface,
are spectrally stable with respect to pearling if and only if the scaled Lagrange multiplier $\hlam$, introduced in (\ref{e:FCH-static}),
satisfies the filament pearling stability condition
\beq\label{e:pearl-cond-pr}
 \hlam S_\fil + (\eta_1-\eta_2)\left(\norm{\hpsi^\prime_{\fil,0}}_{L^2(\RR)}^2+\lambda_{\fil,0}\norm{\hpsi_{\fil,0}}_{L^2(\RR)}\right) <0,
\eeq
where $\psi_{\fil,0}$ is the ground state eigenfunction of the operator $L_{\fil,0}$ introduced in (\ref{e:Lf0-def}) with eigenvalue $\lambda_{\fil,0}>0.$
\end{theo}

In section \ref{sec:coordinates} we introduce the local codimension one and two coordinate systems, define the admissible interfaces and
hypersurfaces, and construct the corresponding bilayer and filament morphologies via the dressing process.
In sections \ref{Sec:Pearling-Bl} and \ref{Sec:Pearling-Pr} we characterize the pearling spectra of the linearizations about the corresponding bilayer and filament morphologies, respectively.

\section{Local coordinate systems and the dressing process}
\label{sec:coordinates}

The stability and slow evolution of localized solutions within the FCH gradient flow (\ref{e:FCH-eq}) is strongly influenced by the value of the far-field chemical potential, $\mu$. Within
the construction of the bilayer and filament morphologies we parameterize the far-field chemical potential by the
scaled Lagrange-multiplier $\hlam$ introduced in (\ref{e:FCH-static}). Indeed, under the FCH gradient flow (\ref{e:FCH-eq}) the interactions
between spatially localized structures are not dominated by exponentially weak tail-tail interactions, but rather are mediated
through the dynamics of $\hlam$. Since the geometric motion of the interfaces and the evolution of $\hlam$ occurs on the
relatively slow $O(\eps^{-1})$ time scale, and the pearling instabilities are manifest on the quick $O(\eps)$
time-scale, it is self-consistent to view the bilayer and filament morphologies as static on the time-scale of
bifurcation. Indeed we rewrite (\ref{e:FCH-static}) as a vector system in which the $\hlam$ scaling is natural
\beqa
\eps^2 \Delta u -W'(u) &=& \eps v, \label{QE-1}\\
(\eps^2\Delta -W''(u))v & =& \left(-\eps^2\eta_1\Delta u+\eta_2 W'(u)\right)+\hlam.\label{QE-2}
\eeqa
The bilayer and filament morphologies render the residual $\mu(u)-\eps\hlam = O(\eps^2)$ in the $L^2(\Omega)$ norm and are constructed from (\ref{QE-1})-(\ref{QE-2}) as a perturbation of the quasi-equilibria of the associated Cahn-Hilliard free energy, that is solutions of
\beq
\label{e:CH-eq}
  \eps^2\Delta u -W'(u) = O(\eps).
\eeq
The bilayer and filament morphologies, defined more precisely in the sequel, are constant off of the reaches of the respective
underlying interface and hypersurface, and limit to a common far-field value parameterized by $\hlam$. We construct the bilayers and filaments
separately, and piece them together additively so long as the respective reaches are disjoint.

\subsection{Admissible codimension-one interfaces and their dressings}

Given a smooth, closed $d-1$ dimensional manifold $\Gamma_b$ immersed in $\Omega\subset\RR^d$,
we define the local  ``whiskered''  coordinates system in a neighborhood of $\Gamma_b$ via the mapping
\beq
\label{e:codim1-cov}
 x=\rho(s,z):= \zeta_b(s)+\eps \nu(s)z,
\eeq
where $\zeta_b:S\mapsto\RR^d$ is a local parameterization of $\Gamma_b$ and $\nu(s)$ is the outward unit normal to $\Gamma_b.$
The variable $z$ is often called the $\eps$-scaled, signed distance to $\Gamma_b$, while the variables $s=(s_1,\ldots, s_{d-1})$ parameterize
the tangential directions of $\Gamma_b$.
\begin{Def}\label{def:admissible}
For any $K, \ell > 0$ the family, $\cG_{K,\ell}^b$,  of admissible interfaces is comprised of closed (compact and without boundary),
oriented $d-1$ dimensional manifolds $\Gamma_b$  embedded in $\mbbR^d$, which are far from self-intersection and with a smooth second fundamental form.  More precisely,
\vskip 0.01in
\begin{tabular}{lp{5.0in}}
 (i)& The  $W^{4,\infty}(S)$ norm of the 2nd Fundamental form of $\Gamma_b$ and its principal curvatures are bounded by $K$.\\
 (ii)& The whiskers of length $3\ell < 1/K$, in the unscaled distance, defined for each~$s_0\in S$ by, $w_{s_0}:=\{x: s(x)=s_0, |z(x)|<3\ell/\eps\}$,  neither intersect each-other
nor $\partial\Omega$ (except when considering periodic boundary conditions). \\
(iii)& The surface area, $|\Gamma_b|$, of $\Gamma_b$ is bounded by $K$.
\end{tabular}
\end{Def}

For an admissible codimension-one interface $\Gamma_b$ the change of variables $x \rightarrow \rho(s,z)$ given by (\ref{e:codim1-cov}) is a $C^4$ diffeomorphism
on the reach of $\Gamma_b$, defined as the set
\beq
\label{e:whiskNbd}
\Gamma_b^\ell := \left\{ \rho(s,z) \in \mbbR^d \Bigl| s \in S, -\ell/\eps \le z \le \ell/\eps \right\}\subset\Omega.
\eeq
On the reach we may view $x=x(s,z)$ and equivalently $(s,z)=(s(x),z(x))$; within the whiskered coordinate system
the Cartesian Laplacian takes the form
\beq\label{e:LaplacianBl}
 \eps^2\Delta_x = \partial_z^2 + \eps (\partial_zJ)/(\eps J)\partial_z + \eps^2 J^{-1} \sum\limits_{i,j=1}^2 \frac{\partial}{\partial s_i}G^{ij}J\frac{\partial}{\partial s_j}=\partial_z^2+\eps H(s,z)\partial_z +\eps^2 \Delta_G,
\eeq
where $J$ is the Jacobian of the change of variables,
\beq
\label{e:extcurv}
 H:=\partial_zJ/(\eps J)=H_0(s)+\eps z H_1(s) +O(\eps^2),
\eeq
is the extended curvature, given at leading order by the mean curvature $H_0=H_0(s)$,
 and ${\bf G}=G_{ij}$ is the metric tensor whose inverse has components $G^{ij}$. The operator $\Delta_G$ takes the form
\beq\label{e:DeltaG}
\Delta_G = \Delta_s +\eps zD_{s,2},
\eeq
where $\Delta_s$ is the usual Laplace-Beltrami operator on $\Gamma_b$ and~$D_{s,2}$ is a relatively bounded perturbation of~$\Delta_s$ on~$H^2_c(\Gamma_{b}^\ell)$, the
subset of $H^2(\Omega)$ comprised of functions with compact support within~$\Gamma_{b}^\ell$.
In particular we have the expansion
\beq \label{def-D2}
D_{s,2} = \sum_{i,j=1}^d d_{ij}(s,z) \frac{\partial^2}{\partial_{s_i}\partial_{s_j}} + \sum_{j=1}^d d_j(s,z) \frac{\partial}{\partial_{s_j}},
\eeq
where the coefficients satisfy the bounds
\beq \label{e:HP-est}
 \max_{ij} \left( \| \partial_z^m \nabla_s^k d_{ij}\|_{L^\infty(\Gamma^\ell_b)}, \|\partial_z^m \nabla_s^kd_j\|_{L^\infty(\Gamma^\ell_b)} \right) \leq C\eps ^m,
 \eeq
 for $m,~k\in\{0, 1, 2, 3, 4\}$ and $C>0$ depending only upon the choice of $\cG_{K,\ell}^b$, see section 6 of \cite{HAYRAPETYAN2014SPECTRA} for details.

We factor the Jacobian as $J=J_0(s)\tJ(s,z)$ where
\beq\label{def-tJ}
 \tJ(s,z):= \eps \prod\limits_{i=1}^{d-1}(1-\eps z k_i(s)) = \sum_{j=0}^d \eps^{j+1}z^j K_j(s),
\eeq
and $K_j$ is the sum of the $j$'th order monomials in the curvatures.
With this factored form, the Laplace-Beltrami operator is self-adjoint in the $J_0$ weighted integral over $\Gamma_b$, and for any $f,g\in L^2(\Omega)$ with
support inside the reach, $\Gamma^\ell_b$, of $\Gamma_b$ we write
\beq
\left(f,g\right)_{L^2(\Omega)} = \int_{\Gamma_b}\int_{-\ell/\eps}^{\ell/\eps} f(s,z)g(s,z)J(s,z)dz\, ds = \int_{\Gamma_b} (f,g)_\tJ(s) J_0(s) ds,
\eeq
where we have introduced the inner product
\beq
\label{def-tJip}
(f,g)_\tJ(s):= \int_{-\ell/\eps}^{\ell/\eps} f(s,z)g(s,z)\tJ(s,z)\, dz.
\eeq

\begin{definition} Given an admissible codimension-one interface $\Gamma_b\in\cG_{K,\ell}^b$ and $f:\RR\rightarrow\RR$ which tends to constant value~$f_{\infty}$ at an
exponential rate as~$z\rightarrow\pm\infty$,  then we define the $H^2(\Omega)$ function
\beq
\label{def:bl-dress}
    f_{\Gamma_b}(x):=\left\{ \begin{array}{cl} f(z(x))\chi(|z(x)|/\ell)+f_\infty(1-\chi(|z(x)|/\ell)) & x\in\Gamma_b^\ell \\
                                                                                       f_\infty & x\in\Omega\backslash\Gamma_b^\ell
                                                                                       \end{array} \right.,
\eeq
where~$\chi:\RR\rightarrow\RR$ is a fixed, smooth cut-off function which takes values one on~$[0, 1]$ and $0$ on $[2,\infty)$.
We call $ f_{\Gamma_b}$ the {\bf dressing of~$\Gamma_b$ with~$f$},  and by abuse of notation will drop the $\Gamma_b$ subscript
when doing so creates no confusion.
\end{definition}

In the whiskered coordinates, the Cahn-Hilliard Euler-Lagrange relation~(\ref{e:CH-eq}) reduces at leading order to a second-order ODE in $z$,
for the one-dimension profile~$\phi(z)$, given in (\ref{e:BLcp}). Since the double-well $W$ has unequal depth wells $0=W(b_-)>W(b_+)$,
a simple phase-plane analysis shows that this equation supports a unique solution $\phi_b$ which is {\em homoclinic} to $b_-$ as $z\to\pm\infty$. To
each admissible codimension-one interface $\Gamma_b$ we associate the bilayer dressing $U_b(\cdot;\Gamma_b)$ of $\Gamma_b$ with $\phi_b$ as defined
by (\ref{def:bl-dress}).
To form the bilayer morphology we incorporate the $O(\eps)$ corrections
to $u_b$, as in (\ref{e:ub-exp})
where $U_{b,1}$ is chosen to render the chemical potential $\mu(u_b)=O(\eps^2).$
Focusing on the reach, $\Gamma_b^\ell$ of the interface and
inserting the expansions (\ref{e:LaplacianBl}) and (\ref{e:ub-exp}) into (\ref{e:mu}) we find that
$U_{b,1}(x)=\phi_{b,1}(z(x))$ where
\beq
  L_{b,0}^2 \phi_{b,1} = -\eta_d W'(\phi_b) +\hlam = -\eta_d \partial_z^2 \phi_b +\hlam,
  \eeq
and the linearization $L_{b,0}$ of (\ref{e:BLcp}) about $\phi_b$ is given by (\ref{e:L0-def}).
The linearization of (\ref{e:BLcp}) about $u_b$ is denoted
\beq\label{e:L1-def}
L_{b,1} := \partial_z^2 -W''(U_b+\eps U_{b,1}),
\eeq
and the function $\Phi_{b,j}:=L_{b,0}^{-j} 1$ for $j=1, 2$ whose associated $\Gamma_b$-dressing
leads to expression
\beq \label{e:Ub1}
U_{b,1} = \hlam \Phi_{b,2} - \eta_d L_{b,0}^{-1}\left(\frac{z}{2}\partial_z \phi_b\right).
\eeq
Here, and in the sequel, we treat the inversion of the 1D operator $L_{b,0}$ as if it acts on the natural extension of the corresponding function to $L^2(\RR).$
The difference between these inversions and those based upon a finite domain $[-\ell/\eps,\ell/\eps]$ subject to Neumann boundary conditions
is on the order of $O(e^{-\ell\alpha_-/\eps})$ and is immaterial to our analysis, see \cite{chen1994spectrum} for  detailed discussion
in the context of single-layer solutions of the Allen-Cahn equation.
\subsection{Admissible codimension-two hypersurfaces and their dressings}\label{sec:coord-pr}

We construct filament morphologies of the FCH free energy by the dressing of an admissible codimension-two manifold immersed in $\Omega\subset\RR^3$.
The construction is based upon a foliation of a neighborhood of a smooth, closed, non-self intersecting one dimensional manifold $\Gamma_\fil$ immersed in $\Omega$
and parameterized by $s\in S\mapsto\zeta_\fil(s)\in\Omega$. The whiskered
coordinate system takes the form
\beq
  x = \rho_\fil(s,z_1,z_2) = \zeta_\fil(s)+\eps\left(z_1N_1(s)+z_2N_2(s)\right),
\eeq
where $N_1(s)$ and $N_2(s)$ are orthogonal unit vectors which are also orthogonal to the tangent vector $\zeta_\fil^\prime(s)$, defined by
\beq
    \frac{\partial \bN^i}{\partial s} = -\kappa_i\bT,\quad i=1,2,
\eeq
where
\beq\label{CS-eq:VecKappa}
    \vec{\kappa}(s,z) := (\kappa_1,\kappa_2)^T,
\eeq
is the normal curvature vector with respect to~$\{\bN^1,\bN^2\}$ and $z=(z_1,z_2)$. The Jacobian associated to the change of variables takes the form
\beq\label{e:PrJ}
    J_\fil = \eps^2\tilde J_\fil,
\eeq
where~$\tilde J_\fil:=1-\eps z\cdot\vkappa$.

\begin{Def}\label{def:admissible-2}
For any $K, \ell > 0$ the family, $\cG_{K,\ell}^\fil$,  of admissible hypersurfaces is comprised of closed (compact and without boundary),
oriented $1$ dimensional manifolds $\Gamma_\fil$  embedded in $\mbbR^3$, which are far from self-intersection and with a smooth second fundamental form.  More precisely,
\vskip 0.01in
\begin{tabular}{lp{5.0in}}
 (i)& The  $W^{4,\infty}(S)$ norm of the 2nd Fundamental form of $\Gamma_\fil$ and its principal curvatures are bounded by $K$.\\
 (ii)& The whiskers of length $3\ell < 1/K$, in the unscaled distance, defined for each~$s_0\in S$ by, $w_{s_0}:=\{x: s(x)=s_0, |z(x)|<3\ell/\eps\}$,  neither intersect each-other
nor $\partial\Omega$ (except when considering periodic boundary conditions). \\
(iii)& The length, $|\Gamma_\fil|$, of $\Gamma_\fil$ is bounded by $K$.
\end{tabular}
\end{Def}
Defining the reach, $\Gamma_\fil^\ell$, of $\Gamma_\fil$ in a manner analogous to (\ref{e:whiskNbd}), then within the reach
the Laplacian admits the local form
 \beq\label{e:PrClosedLap}
    \eps^2\Delta_x= \Delta_z -\eps D_z+\eps^2\partial_G^2,
\eeq
where we introduce the operators
\begin{align}
  D_z:= & \frac{\vkappa}{\tilde J_\fil}\cdot\nabla_z, \label{e:Dz}\\
  \partial_G^2:= &\frac{1}{\tJ_\fil}\left(\partial_s \left(\frac{1}{\tJ_\fil}\partial_s\right)\right)= \frac{1}{\tilde J_\fil^2}\partial_s^2+\eps\frac{z\cdot\partial_s\vkappa}{\tilde J_\fil^3}\partial_s.\label{e:partialG2}
\end{align}
For any $f,g\in L^2(\Omega)$ with support inside the reach, $\Gamma^\ell_\fil$, of $\Gamma_\fil$ we write
\beq
\left(f,g\right)_{L^2(\Omega)} = \int_{\Gamma_\fil}\int_{0}^{\ell/\eps} f(s,z)g(s,z)\eps^2 \tJ_\fil(s,z)dz\, ds = \eps^2 \int_{\Gamma_\fil} (f,g)_{\tJ_\fil}(s) ds,
\eeq
where we have introduced the filament inner product
\beq
\label{def-Jipfil}
(f,g)_{\tJ_\fil}(s):= \int_{0}^{\ell/\eps} f(s,z)g(s,z)\tJ_\fil(s,z)\, dz.
\eeq

When acting on functions with radial symmetry with respect to the codimension-two filament $\Gamma_f$ it
is convenient to write the Laplacian in the equivalent form
\beq
\eps^2\Delta_x = \Delta_R-\eps D_z+\eps^2\partial_G^2.
\eeq
where $\Delta_R$ is the usual polar Laplacian in $(R,\theta)$ corresponding to the scaled normal distances $z=(z_1,z_2),$
see \cite{DAI2015COMPETITIVE} for further details on these coordinate changes.

\begin{definition} Given an admissible codimension-two filament $\Gamma_\fil\in\cG_{K,\ell}^\fil$ and a smooth function~$f:\RR_+\rightarrow\RR$ which tends to constant
value~$f_\infty$ at an $O(1)$ exponential rate as~$R\rightarrow\infty$, we define the $H^2(\Omega)$ function

\beq
    f_{\Gamma_\fil}(x):=\left\{
               \begin{array}{cl}
                 f(z(x))\chi(|R(x)|/\ell)+f_\infty(1-\chi(|R(x)|/\ell)) & x\in\Gamma_\fil^\ell \\
                f_\infty & x\in \Omega\backslash \Gamma_\fil^\ell.
                 \end{array}\right.,
\eeq
where~$\chi:\RR\rightarrow\RR$ is a fixed, smooth cut-off function taking values one on~$[0, 1]$, and zero on $[2,\infty)$.
We call $f_{\Gamma_\fil}$ the {\bf dressing of $\Gamma_\fil$ with~$f$}, and by abuse of notation we will drop the $\Gamma_\fil$ subscript when doing so creates no confusion.
\end{definition}

In the codimension-two whiskered coordinates, the Cahn-Hilliard Euler Lagrange relation~(\ref{e:CH-eq}) reduces at leading order to a second-order
non-autonomous ODE in $R$, for the one-dimension profile~$\phi_\fil(R)$,
\beq
\label{e:FLcp}
  \partial_R^2 \phi_\fil+\frac{1}{R}\partial_R\phi_\fil - W'(\phi_\fil)=0,
\eeq
to which we impose the boundary conditions $\partial_R \phi_\fil(0)=0$ and $\phi_\fil\to b=b_-+\eps\gamma_1+O(\eps^2)$ as $R\to\infty$.
We denote by $U_\fil$ the dressing of $\Gamma_\fil$ by $\phi_\fil$, and mirroring the steps for the bilayer morphology,
we construct the filament morphology
\beq
\label{e:PR}
u_\fil(x;\Gamma_\fil,\hlam):= U_\fil(x)+\eps U_{\fil,1}(x),
\eeq
where the correction term takes the form
\beq \label{e:Ub1fil}
U_{\fil,1} = \hlam \Phi_{\fil,2} - \eta_d L_{\fil,0}^{-1}\left(\frac{R}{2}\phi_\fil'\right).
\eeq
As in the codimension-one case, we have introduced the operator
\beq \label{e:Lf0-def}
L_{\fil,0} := \partial_R^2 +\frac{1}{R}\partial_R-W''(\phi_\fil),
\eeq
corresponding to the linearization of (\ref{e:FLcp}) about $\phi_{\fil}$ and the functions
$\Phi_{\fil,j}:=L_{\fil,0}^{-j} 1$ for $j=1, 2$ and their $\Gamma_{\fil}$ dressings. Here, and in the sequel, the inversion of $L_{\fil,0}$ is understood
as acting on the natural extension it argument to $L^2(\RR_+).$  We also introduce
\beq
\label{e:Lf0-uf-def}
L_{\fil,0}^{u_\fil}:=  \partial_R^2 +\frac{1}{R}\partial_R-W''(u_\fil),
\eeq
the linearization of (\ref{e:FLcp}) about $u_\fil.$
In particular $u_\fil$ admits the far-field value
\beq \label{e:uf-ff}
u_{\fil}(x) = b_-+\eps \Frac{\hlam}{\alpha_-^2}, \quad {\rm for}\, x\in \Omega\backslash\Gamma_\fil^{3\ell}.
\eeq
Unlike in the codimension-one regime, the radial symmetry of the filament structures does not extend to the
eigenfunctions, which may have a non-trivial $\theta$ dependence. Consequently, we will have need for the full operators
\beq \label{e:Lf-def}
L_{\fil} := \Delta_z-W''(\phi_\fil),
\eeq
and
\beq
\label{e:Lf-uf-def}
L_{\fil}^{u_\fil}:=  \Delta_z-W''(u_\fil),
\eeq
and their restrictions to the certain invariant subspaces.


\section {Pearling eigenvalues of bilayer morphologies}\label{Sec:Pearling-Bl}

 The second variational derivative of $\cF$ at a generic function $u$ was introduced in (\ref{e:LL}). When $u$ is a bilayer morphology $u_b$
associated to an admissible codimension-one interface  $\Gamma_b$, as defined in (\ref{e:ub-exp}), the second variational derivative, denoted $\LL_b$,
takes a simplified form when acting on functions $u\in H^4(\Omega)$ whose support lies within the reach, $\Gamma_b^\ell$, of $\Gamma_b$.
On this subspace the operator admits the exact expression
$$
\bL_{b}= \left(L_{b,1} +\eps H\partial_z +\eps^2 \Delta_G\right)^2 +\eps \left(\eta_1(\partial_z^2+\eps H\partial_z+ \eps^2\Delta_G) -\eta_2 W''(u_b)\right)
- \left(\partial_z^2u_b-W'(u_b)+\eps H\partial_zu_b\right)W'''(u_b),
$$
where $L_{b,1}$ is defined in (\ref{e:L1-def}). The two dominant operators in $\LL_b$ are $\partial_z^2$ and $\eps^2\Delta_s$, with the Laplace Beltrami operator forming the principle part of $\Delta_G$,
see (\ref{e:DeltaG}).
This observation suggests the introduction of
\beq
\label{def:cLb}
\cL_b:=L_{b,0}+ \eps^2\Delta_s,
\eeq
which balances the sum of the mostly negative Sturm-Liouville operator $L_{b,0}$, introduced in (\ref{e:L0-def}), and the non-positive
 Laplace-Beltrami operator associated to $\Gamma_b$. With this notation, we may write $\bL_b$ as the square of
$\cL_b$ plus lower order terms
\beq\label{PB-mbbL}
	\LL_b = \cL_b^2+\eps \LL_1+ O(\eps^2),
\eeq
where the first correction term takes the form
\begin{align}
	\LL_1 &:= \cL_b\circ (H_0\partial_z-W'''(U_b)U_{b,1})+(H_0\partial_z-W'''(U_b)U_{b,1})\circ \cL_b + \eta_1 \partial_z^2 -\eta_2 W''(U_b) \label{eq:L1}\\
	     	  & \quad   -(L_{b,0} U_{b,1}+H_0 U_b')W'''(U_b).\nonumber
\end{align}
The operator $\cL_b$ is self-adjoint in the $J_0$ weighted inner product on the reach of $\Gamma_b$, and the unbounded terms in the operators~$\LL_i$
for $i\geq1$ are relatively compact in $H^1_0(\Gamma_b^\ell)$ with respect to $\cL_b^2$.
In this section we analyze the spectrum of $\LL_b$ on spaces where the positive spectra of $L_{b,0}$ balance the negative spectra of $\eps^2\Delta_s$.
\begin{definition}\label{CS-def:ScaledEigenfunctions}
We define the \textbf{scaled eigenfunctions}~$\spsi_{b,k}:=\chi(z)\tilde J^{-1/2}\hpsi_{b,k}$, where~$\hpsi_{b,k}$ is the~$k^{th}$ eigenfunction of~$L_{b,0}$ and~$\chi(z)$ is a
$\cC^\infty$ cut-off function that takes values $1$ for $|z|<\ell/(2\eps)$, $0$ for $|z|>\ell/\eps$ and is smooth and monotone between.
\end{definition}

In \cite{Doelman2014meander}, the localized spectrum of the linearization about FCH equilibria arising as the dressing of single-curvature interfaces was presented.
We extend these results to include the bilayer morphologies arising as the dressing of admissible codimension one interfaces,
which also incorporate perturbations of the back-ground state, focusing on the pearling eigenvalues that are the dominant modes of
instability. For an admissible codimension-one interface,~$\Gamma_b$, we consider the eigenvalue problem
\beq\label{e:EVP}
    \LL_b\Psi_{b} = \Lambda_{b}\Psi_{b},
\eeq
associated to the second variation, $\LL_b$ of $\cF$ about a bilayer morphology $u_b$.
The spectrum of $\LL_b$ cannot be localized by a regular perturbation expansion since the eigenvalues are asymptotically
close together. A perturbation analysis requires bounds on the spectrum that are uniform in $\eps\ll 1$. To this end we recall the
tensor product formulation of the pearling and meander eigenmodes, (\ref{e:BL-tensor}), and introduce the
$L^2(\Omega)$ orthogonal projection~$\Pi$ onto the space
\beq\label{e:spaceCorrespond2SmallEV}
X_b(\Gamma_b) := {\rm span} \left\{ \spsi_j(z) \Theta_n(s) \bigl | \, j=0, 1, \,\, {\rm and}\,\, n\in \Sigma_{b,j}(r)\,{\rm respectively}\right \},
\eeq
which approximates the eigenspaces of $\LL_b$ corresponding to pearling $(j=0)$ and meander $(j=1)$ eigenmodes.
We denote the dimension of $X_b$ by $N_b=N_{b,0}+N_{b,1},$ see (\ref{e:SigmaBl}), and remark that the basis of $X_b$
is orthonormal in $L^2(\Omega)$, up to exponentially small terms. Since the basis elements are localized on $\Gamma_b$ their inner product can be written as
\beq
\left(\spsi_k\Theta_n, \spsi_j\Theta_m\right)_{L^2(\Omega)}=\int_{\Gamma_b}\int_{-\frac{\ell}{\eps}}^\frac{\ell}{\eps} \spsi_j(z,s)\spsi_k(z,s)\Theta_n(s)\Theta_m(s)J(z,s)dz\,ds,
\eeq
and from Definition\,\ref{CS-def:ScaledEigenfunctions} of the scaled eigenfunctions and the factored form (\ref{def-tJ}) of the Jacobian
we have
\beq
\left(\spsi_k\Theta_n, \spsi_j\Theta_m\right)_{L^2(\Omega)} = \left(\int_{\Gamma_b}\Theta_n(s)\Theta_m(s) J_0(s)\,ds\right)
                                                                                \left(\int_{-\frac{\ell}{\eps}}^\frac{\ell}{\eps}\hpsi_j(z)\hpsi_k(z)\,dz\right) = \delta_{nm}\delta_{jk}
                                                                                +O(e^{-\ell/\eps}).
\eeq

To address the eigenvalue problem (\ref{e:EVP}), we fix an admissible codimension-one interface, $\Gamma_b\in\cG^b_{K,\ell}$ and expand $\LL_b$
as in  (\ref{PB-mbbL}). We localize the pearling and geometric eigenvalues of $\LL_b$  via an analysis of its projection onto $X_b$,
searching for solutions of the eigenvalue problem,~(\ref{e:EVP}), via the decomposition
\beq
\label{e:EVP-decomp}
    \Psi_b = v_b+ \Psi_b^\perp,
\eeq
where
\beq
\label{e:v-alpha}
v_b=\sum_{k\in\Sigma_{b,0}(r)}\alpha_{0,k}\spsi_{b,0}\Theta_k+\sum_{k\in\Sigma_{b,1}(r)}\alpha_{1,k}\spsi_{b,1}\Theta_k \in X_b
\eeq
and  $\Psi_{b}^\perp\in X_b^\bot$.  We use the projections $\Pi$ and $\tilde{\Pi}=I-\Pi$ to decompose the operator~$\LL_b$ into the $2\times2$ block form,
\beq\label{PB-BreakL}
    \begin{bmatrix}
        M & B           \\[0.3em]
        B^T & \LL^\perp
      \end{bmatrix},
\eeq
where
\begin{align}
\label{def-BreakL}
    M:=\Pi\LL_b\Pi,\quad
    B:=\Pi\LL_b\tilde{\Pi},\quad
    \LL^\perp:=\tilde \Pi\LL_b\tilde \Pi.
\end{align}
From \cite{HAYRAPETYAN2014SPECTRA}, the restricted operator $\LL^\perp$ is $L^2(\Omega)$ coercive on $X_b^\bot$.
In section~\ref{PB-sec:BoundM} we analyze the spectrum of $M$, and show that its spectrum coincides to
leading order with the small eigenvalues of $\LL_b$.

\subsection{Eigenvalues of the pearling matrix $M:=\Pi\LL_b\Pi$}\label{PB-sec:BoundM}
Denote~$v\in X_b$ by
\beq
\label{e:Xb-expand}
v(s,z)=\overline{v}_0(s)\spsi_{b,0}+\ov_1(s)\spsi_{b,1},
\eeq
where for $j=0, 1$ we have introduced
\beq
\ov_j:=\sum_{k\in\Sigma_{b,j}(r)}\alpha_{j,k}\Theta_k.
\eeq

Representing the coefficients of $v$ by  $\valpha_j=(\alpha_{j,k})_{k\in\Sigma_{b,j}(r)}$, for $j=0, 1$ and $\valpha=(\valpha_0,\valpha_1)^T$,
 the action of $\LL_b$  on $v$ can be represented by
the matrix $M\in\mbbR^{N_b\times N_b}$  with entries
\beq\label{PB-Def:M}
    M_{jk}:=(\LL_b\spsi_{b,I(j)}\Theta_j,\spsi_{b,I(k)}\Theta_k)_{L^2(\Omega)},
\eeq
where the index function $I(j)$ takes the value $k$ if $j\in\Sigma_{b,k}(r)$ for $k=0, 1$.
From the expansion of~$\LL_b$, (\ref{PB-mbbL}), we may fix $q\in\NN_+$ and
group terms in $M$ into two classes
\beq\label{PB-def:M}
    M= M^0 +\eps^q\tilde M,
\eeq
where
\begin{align}
    \label{PB-M0}M^0_{jk}&=(\cL_b^2\spsi_{b,I(j)}\Theta_j,\spsi_{b,I(k)}\Theta_k)_{L^2(\Omega)}+
     \sum_{i=1}^q\eps^i(\LL_i\spsi_{b,I(j)}\Theta_j,\spsi_{b,I(k)}\Theta_k)_{L^2(\Omega)},\\
    \tilde M_{jk}&=\sum_{i\geq q}\eps^{(i-q)}(\LL_i\spsi_{b,I(j)}\Theta_j,\spsi_{b,I(k)}\Theta_k)_{L^2(\Omega)}.
\end{align}

The following Lemma shows that for space dimension $d\leq 3$, $r>0$, and $q=2$, an $O(1)$ bound on the $l^\infty(N_b\times N_b)$
norm of $\tilde{M}$ implies a $o(\eps)$ bound on the $l^2(N_b)\mapsto l^2(N_b)$ operator norm, denoted $l^2_*$, of $\eps^q\tilde{M}.$

\begin{lemma}\label{l-MCor}
The dimension $N_b$ of $X_b$ scales as  $N_b\sim \eps^{(1-d)(1-\frac{r}{4})}$. Moreover, there exists $C>0,$ independent of $\eps>0$, such that
for any matrix $ A\in \RR^{N_b\times N_b}$ we have the operator norm bound
   \beq       \|A\|_{l^2_*}  \leq C \eps^{-q_*} \|A\|_{l^\infty(\RR^{N_b}\times \RR^{N_b})}, \eeq
where  $q_*(d,r):= \frac{d-1}{2}(1-\frac{r}{4}).$
\end{lemma}
\begin{proof}
    This is the concatenation of the usual  $l^\infty(N_b\times N_b)$ to  $l^2_*$ estimate and the Weyl asymptotics which control the size of $N_b.$
  \end{proof}

We further divide $M^0$ into sub-blocks
\beq\label{PB-M0-sub}
M^0 =    \begin{bmatrix}
        M^{0,0} & M^{0,1}           \\[0.3em]
        M^{1,0} & M^{1,1}
      \end{bmatrix},
\eeq
with the $M^{j,k}$ sub-block corresponding the inner products with entries from $\Sigma_{b,j}(r)$ and $\Sigma_{b,k}(r)$ for $j,k=0,1.$
The $M^{0,0}$ and $M^{1,1}$ sub-blocks are called the pearling and geometric sub-blocks respectively.

The following proposition characterizes the leading order entries of $M^0$ in terms of system dependent parameters
\begin{prop}
\label{p:M0}
For an admissible bilayer morphology $u_b(\cdot; \Gamma_b, \hlam)$ with far-field parameter $\hlam$,  the entries of the pearling, geometric,
and cross-term sub-blocks of $M^0$ defined in (\ref{PB-M0-sub}) take the form
\beq \label{e:M00-asymp}
    M^{0,0}_{jk}=
    \begin{cases}
        \eps\left(P_{k,0}^2-\hlam S_b-\eta_d\lambda_{b,0}\|\hpsi_{b,0}\|_2^2\right) + O(\eps\sqrt{\eps}) &~\text{ if }j=k,\\
      - \eps^2 \int_{\Gamma_b} (\|\partial_z\hpsi_{b,0}\|_{L^2}^2 K_1^2 +S_{1,0} \Hbl_1)\Theta_k\Theta_jJ_0\,ds
        +O(\eps^2\sqrt{\eps})&~\text{ if }j\neq k,
    \end{cases}
\eeq
for $j,k\in\Sigma_{b,0}(r),$ where $S_b$ and $S_{1,0}$ are given in (\ref{e:S-bl}) and (\ref{e:S1j-def}) respectively,
\beq \label{e:M11-asymp}
    M^{1,1}_{jk} =
    \begin{cases}
        \eps P_{k,1}^2 + O(\eps^2) &~\text{ if }j=k,\\
       - \eps^2 \int_{\Gamma_b} (\|\partial_z\hpsi_{b,1}\|_{L^2}^2 K_1^2 +S_{1,1} \Hbl_1)\Theta_k\Theta_jJ_0\,ds
         +O(\eps^2\sqrt{\eps})&~\text{ if }j\neq k,
    \end{cases}
\eeq
for $j,k\in\Sigma_{b,1}(r),$ where $S_{1,1}$ is given in (\ref{e:S1j-def}), and
\beq
    M^{0,1}_{jk} =    \eps S_2 \int_{\Gamma_b} H_0 \Theta_j\Theta_k J_0\, ds + \cL_{jk}+
    O(\eps^2),
\eeq
for $j\in\Sigma_{b,0}(r)$, $k\in\Sigma_{b,1}(r),$ where $S_2$ is given in (\ref{e:S2-def}) and $\cL_{jk}=O(\eps^{\frac32})$ is given in (\ref{e:cLb-asymp}).
 The mean and quadratic curvatures $H_0$ and $H_1$ of $\Gamma_b$ are defined in (\ref{e:extcurv}), $K_1$ is defined in (\ref{def-tJ}),
while the detuning constants,
\beq\label{PB-Pk}
    P_{k,I(k)} := \eps^{-1/2}(\lambda_{b,I(k)}-\eps^2\beta_k),
\eeq
arise in the definition of $\Sigma_{b,I(k)}(r)$ in (\ref{e:SigmaBl}).
\end{prop}
\begin{proof}
The scaled eigenfunctions have support within $\Gamma_b^\ell$ and hence we may change to the whiskered coordinates in the first term of (\ref{PB-M0}). Using the factored form of the Jacobian, and the self adjointedness of $\cL_b$  in the $J_0$ weighted inner product we have
\beq
 \cL_{jk}:=(\cL_b^2\spsi_{b,I(j)}\Theta_j,\spsi_{b,I(k)}\Theta_k)_{L^2(\Omega)} =
 \int_{\Gamma_b}\int_{-\ell/\eps}^{\ell/\eps} \cL_b\left( \Frac{\hpsi_{b,I(j)}\Theta_j}{\sqrt{\tJ}}\right) \cL_b\left( \hpsi_{b,I(k)} \Theta_k \sqrt{\tJ}\right)\, dz\,J_0(s) ds.
 \eeq
 To expand this we deduce from (\ref{def-tJ}) that $\tJ^p$ is $O(\eps^p)$ for any $p\in\RR$, while
\begin{align}
 \partial_z \tJ & =\eps^2 K_1+O(\eps^3), \\
 \nabla_s \tJ& =\eps^2z \nabla_s K_1+O(\eps^3),\\
\Delta_s \tJ & =\eps^2 z\Delta_s K_1+O(\eps^3),\\
\partial_z^2\tJ & = 2\eps^3K_2+O(\eps^4).
\end{align}
Moreover
 \beq
 \cL_b \hpsi_{b,I(j)}\Theta_j = (\lambda_{b,I(j)}-\eps^2 \beta_j)\hpsi_{b,I(j)}\Theta_j = \eps^{\frac12} P_{j,I(j)} \hpsi_{b,I(j)}\Theta_j,
\eeq
so we identify the leading order terms
\begin{align}
 \cL_b\left( \Frac{\hpsi_{b,I(j)}\Theta_j}{\sqrt{\tJ}}\right) &= \eps^{\frac12}P_{j,I(j)}  \Frac{\hpsi_{b,I(j)}\Theta_j}{\sqrt{\tJ}}-\eps^\frac{1}{2}\partial_z\hpsi_{b,I(j)}K_1\Theta_j
  +O(\eps^\frac{3}{2}),\\
 \cL_b\left( \hpsi_{b,I(k)}\Theta_k\sqrt{\tJ}\right) &= \eps^{\frac12}P_{k,I(k)}  \hpsi_{b,I(k)}\Theta_k \sqrt{\tJ} + \eps^\frac{3}{2} \partial_z\hpsi_{b,I(k)}K_1\Theta_k +
 O(\eps^{\frac{5}{2}}),
 \end{align}
where the error is in $L^2(\Omega)$ and we used the estimate $\|\nabla_s \Theta_j\|_{L^2(\Gamma_b)} =O(\eps^{-1})$ to bound the $\nabla_s \Theta_j\cdot \nabla_s\tJ$ terms. Combining these,
and using the parity considerations to eliminate the $z$ integrals of $\hpsi \partial_z \hpsi$ we find
\beq
\label{e:cLb-asymp}
\cL_{jk}=
 \begin{cases}
  \eps P_{j,I}P_{k,I}\delta_{jk}  -\eps^2 \|\partial_z\hpsi_{b,I}\|_{L^2}^2 \int_{\Gamma_b} K_1^2 \Theta_j\Theta_kJ_0\,ds + O(\eps^{\frac{7}{2}}) & \text{if}~I(j)=I(k)=I,\\
  \eps^{\frac{3}{2}}(P_{j,I(j)}+P_{k,I(k)})\int_\RR \hpsi_{b,I(j)}\partial_z \hpsi_{b,I(k)}\,dz \int_{\Gamma_b}K_1\Theta_j\Theta_kJ_0\,ds +O(\eps^\frac52)& \text{if}~ I(j)\neq I(k).
  \end{cases}
 \eeq

 The reduction of the $\LL_1$ inner products are considered case by case. For $j=k\in\Sigma_{b,0}(r)$ we only require leading-order terms.
 Since the leading order term in the Jacobian, $\tJ$, is constant, the calculation reduces to the constant curvature case addressed in
 Lemma 4.3 and Corollary 4.9 of  \cite{Doelman2014meander},
 \beq
  \left( \LL_1 \spsi_{b,I(k)}\Theta_k, \spsi_{b,I(k)}\Theta_k\right)_{L^2(\Omega)} = -(\hlam S_b+\eta_d\lambda_{b,0}\|\hpsi_{b,0}\|_2^2) + O(\sqrt{\eps}),
 \eeq
where we have introduced the 'background factor'
 \beq\label{e:S-bl}
    S_b:=\int_\RR\Phi_{b,1}W'''(\phi_b)\hpsi_{b,0}^2\,dz,
\eeq
which characterizes the impact of the far-field value of $u_b$ on the pearling eigenvalues.
This expansion, in conjunction with (\ref{e:cLb-asymp}) yields the $j=k$ case of (\ref{e:M00-asymp}).
 For $j\neq k$, the $\eps^2\Delta_s$ term in $\LL_b$ induce lower order contributions to the inner product.
This is clear, unless the term falls entirely upon $\Theta_i$ or upon $\Theta_j$,
in which case it becomes $\eps^2\beta_i$ or $\eps^2\beta_j$ (which are $O(1)$) but the integral is lower
order because of orthonormality. The term $\eps^2 \nabla_s\Theta_i\cdot \nabla_s \Theta_j$ might formally appear to be leading order, yet an integration by parts returns us to the prior case.
Therefore, the inner product takes the form
 \beq
  \left( \LL_1 \spsi_{b,I(j)}\Theta_j, \spsi_{b,I(k)}\Theta_k\right)_{L^2(\Omega)} = \int_{\Gamma_b} \left( \LL_1 \frac{\hpsi_{b,I(j)}}{\sqrt{\tJ}},  \frac{\hpsi_{b,I(k)}}{\sqrt{\tJ}}\right)_\tJ \Theta_j\Theta_k J_0(s)\, ds,
 \eeq
which is zero unless the $\tJ$ inner product has non-trivial $s$ dependence. The only leading order term in $\LL_1$ with non-trivial
$s$ dependence is $-H\phi_b^\prime=-H_0(s)\phi_b^\prime(z) + O(\eps)$, however $\phi_b^\prime$ is odd in $z$ so parity issues yield
a non-zero $z$-integral only if $I(j)\neq I(k)$, for which we find
\beq \label{e:M01-est}
 \left( \LL_1 \spsi_{b,I(j)},  \spsi_{b,I(k)}\right)_\tJ =  S_2 H_0(s)  
 + O(\eps), \quad I(j)\neq I(k),
\eeq
where we have introduced the constant
\beq
\label{e:S2-def}
S_2= \int_\RR \hpsi_{b,0}\hpsi_{b,1}\phi_b^\prime W'''(\phi_b)\,dz>0.
\eeq

For $I(j)=I(k)$ the leading order terms occur at next order, and we seek terms which introduce $s$ dependence whose combination
preserves even parity in $z$.  The operator $\LL_1$ can be decomposed into ``even'' operators which preserve $z$ parity and ``odd'' operators
which map odd parity functions to even ones, and conversely. The odd component is the single term
$$[\LL_1]_{\rm odd}=H_0U_b^\prime W'''(\phi_b)+O(\eps^2),$$
while its even component can be further separated into terms, $\LL_1^0$ with no $s$ dependence and a single term
$$[\LL_1]_{\rm even}=[\LL_1^0]_{\rm even}+ [H\phi_b^\prime W'''(\phi_b)]_{\rm even}= [\LL_1^0]_{\rm even}+\eps H_1z\phi_b^\prime W'''(\phi_b) +O(\eps^3).$$
Viewing the integrand in $ \left( \LL_1 \spsi_{b,I(j)},  \spsi_{b,I(k)}\right)_\tJ $ as the action of $\LL_1$ on $\spsi_{b,I(j)}$ subsequently multiplied by $\spsi_{b,I(k)}$ and $\tJ$,  we recall (\ref{def-tJ}) and expand
\beq \label{e:psibJ-exp}
\spsi_{b,I(j)}=\hpsi_{b,I(j)}\tJ^{-\frac12} = \hpsi_{b,I(j)}\left(\eps^{-\frac12} -\frac12 \eps^{\frac12} zK_1(s) +\eps^{\frac32}z^2\left(\frac38 K_1^2(s)-\frac12 K_2(s)\right) + O(\eps^{\frac52})\right).
\eeq
If we use the even part of $\LL_1$ then non-trivial $s$ dependence and overall even parity in $z$ requires either going to third order in one of the three functions, going to second order (odd) in two of the functions, or using the second-order even part of $\LL_1$ on leading order even parts of the three
functions. The first two options introduce two factors of $\eps$, while the third option only introduces one, yielding the leading order term
\beq\label{e:LL1-even}
  \left([\LL_1]_{\rm even} \spsi_{b,I(j)},  \spsi_{b,I(k)}\right)_\tJ = \eps S_{1,I(j)} H_1 + O(\eps^2),
\eeq
where we have introduced the system dependent quantity
\beq
\label{e:S1j-def}
    S_{1,j}:=\int_\RR W'''(\phi_b)\phi_b'\hpsi_{b,I(j)}^2z\,dz,
 \eeq
 for $j=0,1.$  For the odd part of $\LL_1$, overall even $z$ parity requires one odd term from $\spsi_{b,I(j)}$, $\spsi_{b,I(k)}$ or $\tJ$; these
 contribute two copies of $-\frac 12 z K_1$ and one copy of $z K_1$ whose sum cancels, consequently
\beq \label{e:LL1-odd}
  \left([\LL_1]_{\rm odd} \spsi_{b,I(j)},  \spsi_{b,I(k)}\right)_\tJ =O(\eps^2),
\eeq
and combining (\ref{e:LL1-even}) and (\ref{e:LL1-odd}) yields the $j\neq k$ cases of both (\ref{e:M00-asymp}) and (\ref{e:M11-asymp}).
\end{proof}

The diagonal entries, $M^0_{\text{diag}}$, of $M^0$ are $O(\eps)$, and we wish to show that the matrix of off-diagonal entries,
$M^0_{\text{off-diag}}:=M^0-M^0_{\rm diag}$ has an $o(\eps)$ $l^2_*$ norm, so that the eigenvalues of
$M^0$ correspond with its diagonal entries to leading order. From Lemma \ref{l-MCor} we have appropriate operator norm bounds for all
off-diagonal entries of $O(\eps^2\sqrt{\eps})$. The remainder of the off-diagonal entries  can be combined into matrices with $\Theta_j\Theta_k$
factored entries that enjoy bounds independent of the size of the space $X_b.$

\begin{lemma}\label{l:bdM0}
    Let~$\Gamma_b\in\cG^b_{K,\ell}$ be an admissible interface, then its curvatures $\vec k_b\in W^{2,\infty}(S)$. Let~$f:\RR^{d-1}\rightarrow\RR$ be a smooth function, and then there exists $C>0$ such that for any pair of index sets $\Sigma_j\subset\NN_+$ with $|\Sigma_j|=N_j$ for $j=1,2$,
  the matrix~$A\in\RR^{N_1\times N_2}$, defined by its entries
    \beq\label{PB-eq:TheoM}
        A_{ij} = \int_{\Gamma_b} f(\vec k_b)\Theta_i\Theta_j\,J_0\,ds,
    \eeq
    where $\{\Theta_k\}_{k=0}^\infty$ are the eigenfunctions of Laplace-Beltrami operator associated to $\Gamma_b$; satisfies
    \beq\label{l:A-normbdd}
        \norm{A}_{\RR^{N_2}\mapsto \RR^{N_1}}\leq C.
    \eeq
 If moreover $\Sigma_1=\Sigma_{b,0}(r)$ and $\Sigma_2=\Sigma_{b,1}(r)$, 
 then we have the estimate
 \beq
 \label{l:A-normbdd01}
  \norm{A}_{\RR^{N_2} \mapsto \RR^{N_1}}\leq C\eps^s,
 \eeq
 where $s:=\frac{3-d}{2}+r\frac{d}{8}>0,$ for $d\leq 3$ and $r>0$ as in (\ref{e:SigmaBl}).
\end{lemma}
\begin{proof}
The $\RR^{N_2}\mapsto \RR^{N_1}$ norm of $A$ is defined by
\beq\label{PB-eq:OperNormDef}
    \norm{A}_{\RR^{N_2}\mapsto \RR^{N_1}}:=\inf \{c>0~\Big|~ |(A\vv,\vw)|\leq c\norm{\vv}_{l^2}\norm{\vw}_{l^2}, \text{ for all }\vv\in\RR^{N_2},
      \vw\in\RR^{N_1}\}.
\eeq
Let $\vv\in\RR^{N_1}$ and $\vw\in\RR^{N_2}$, using the definition of~$A$,~(\ref{PB-eq:TheoM}), we can write
\beq
\label{PB-eq:Theo1b}
    \left|(A\vv,\vw)\right| = \left|\sum_{i,j}\int_{\Gamma_b} f(\vec{k}(s))\Theta_iv_i\Theta_jw_j\,J_0ds\right| =\left|\int_{\Gamma_b} f(\vec{k}(s))v(s) w(s)\,J_0ds\right|,
\eeq
where $v:=\sum\Theta_iv_i$ and $w:=\sum\Theta_iw_i$  reside in $L^2(\Gamma_b)$. Applying H$\ddot{\text{o}}$lder's inequality to the last integral yields
\beq\label{PB-eq:Theo1}
    \left|(A\vv,\vw)\right|\leq\norm{f(\vec{k})}_{L^\infty(\Gamma_b)}\norm{v}_{L^2(\Gamma_b)}\norm{w}_{L^2(\Gamma_b)}
                             \leq\norm{f(\vk\,)}_{L^\infty(\Gamma_b)}\norm{\vv}_{l^2}\norm{\vw}_{l^2},
\eeq
where the last inequality follows from the orthonormality  of the Laplace-Beltrami eigenfunctions in the $J_0$-weighted ${\Gamma_b}$-inner product.
For a class of admissible interfaces, $\vec{k}$ is uniformly bounded in $L^\infty$, and since $f$ is smooth, we may find a constant $C>0$, depending
only on $K,\ell$ such that $\norm{f(\vec{k})}_{L^\infty(\Gamma_b)}\leq C$.  The result (\ref{l:A-normbdd}), follows.
To establish (\ref{l:A-normbdd01}) we assume $\Sigma_1=\Sigma_{b,0}$ and use the gap between $\beta_i$ for $i\in \Sigma_{b,0}$ and $\beta_j$ for $j\in\Sigma_{b,1}$
to bound the entries $A_{ij}$ directly. Indeed
\begin{align}
\beta_i A_{ij} &= \int_{\Gamma_b} f(\vk\,) \left(\Delta_s \Theta_i\right)\Theta_k J_0\,ds =
                      \int_{\Gamma_b} \Theta_i \Delta_s\left( f(\vk\,)\Theta_k \right)J_0\,ds, \\
                      & =\beta_j A_{ij} + \int_{\Gamma_b}\Theta_i \left(\Delta_s f(\vk\,) \Theta_j + 2\nabla_s f(\vk\,)\cdot \nabla_s \Theta_j\right)J_0\,ds.
\end{align}
However $\|\Delta_sf(\vk\,)\|_{L^\infty}$ and $\|\nabla_s f(\vk\,)\|_{L^\infty}$ are both uniformly bounded, while $\|\nabla_s\Theta_j\|_{L^2(\Gamma_b)}=\sqrt{\beta_j}$,
and we deduce that
\beq
 |A_{ij}|\leq  \frac{C\sqrt{\beta_j}}{\beta_i-\beta_j}\leq C \eps^{1+\frac{r}{8}},
 \eeq
where the last inequality follows from the bounds  $\beta_i\geq \alpha \eps^{-2}$ for $i\in\Sigma_{b,0}$  and $\beta_j< \alpha \eps^{-2+r/4}$ for $j\in\Sigma_{b,1}$ and
some $\alpha>0$ independent of $\eps>0.$
Applying Lemma\,\ref{l-MCor} yields (\ref{l:A-normbdd01}).
\end{proof}

We have established the following result.
\begin{prop}\label{PB-TheoCor}
Fix  a set $\cG^b_{K,\ell}$ of admissible interfaces. Then there exists $\gamma=\gamma(d,r)>1$ such that for all
$\eps$ sufficiently small, the eigenvalues of $M$, defined in (\ref{PB-Def:M}) are given, to $O(\eps^\gamma)$, where
$$\gamma=\frac{5-d}{2} + r\frac{d}{8}>1,$$
by its diagonal elements as indicated in (\ref{e:M00-asymp}). In particular the pearling eigenvalues are given to leading
order by the diagonal elements of $M^{0,0}$ given in (\ref{e:M00-asymp}).
\end{prop}
\begin{proof}
 Lemmas\,\ref{l:bdM0} and \ref{l-MCor}, in conjunction with the estimates in Proposition\,\ref{p:M0} establish an $O(\eps^\gamma)$ bound on the operator
norm of the off-diagonal components of $M$.
\end{proof}

\subsection{Bounds on the off-diagonal operators}\label{PB-sec:BoundOffDiag}

The off-diagonal operator, $B=\Pi\LL_b\tilde\Pi$, is defined in (\ref{def-BreakL}). Recalling the expansion (\ref{PB-mbbL}) of $\LL_b$
into its dominant part $\cL_b^2$ and asymptotically small, relatively bounded perturbations, our  first step is to bound the dominant part,
$B_0:=\Pi\cL_b^2\tilde\Pi$, of $B$. Since, $B_0$, and $B_0^T$ enjoy the same bounds, for  simplicity we address the latter.

\begin{corol}\label{PB-equivnorm}
Fix  a set $\cG^b_{K,\ell}$ of admissible interfaces, then for each $m\in\NN_+$ there exists $C>0$ such that for each $\Gamma_b\in\cG^b_{K,\ell}$ and all $\vv\in\RR^{N_b}$ we have
\beq \label{EN-z}
\|\partial_z^{2m} v\|_{L^2(\Omega)} + \| \eps^{2m}\Delta_s^m v\|_{L^2(\Omega)} \leq C \|\vv\,\|_{l^2},
\eeq
where $v \in X_b(\Gamma_b)$ takes the form (\ref{e:Xb-expand}).
\end{corol}
\begin{proof}
The bound on the first term on the left-hand side of  (\ref{EN-z}) follow from applications of Lemma \ref{l:bdM0} with the choices
 $f_{j,k}(\vk_b) = (\partial_z^{2m}\spsi_{b,j},\partial_z^{2m}\spsi_{b,k})_\tJ$
 for $j,k$ running over the values $0, 1$. For the second term we observe from the form of $v$  that
 \beq
    \eps^{2m}\Delta_s^m v = \sum_{j\in\Sigma_{b}} \eps^{2m}\beta_j^m v_j\Theta_j  \spsi_{b,I(j)},
 \eeq
 and hence $\eps^{2m}\Delta_s^m v$ is localized on $\Gamma_b$ and indeed lies in $X_b$.
 From the form, (\ref{e:SigmaBl}), of $\Sigma_{b,0}$ and $\Sigma_{b,1}$ it follows that $\eps^{2m}\beta_j^m=\lambda_{b,I(j)}^m+O(\eps^{r/2})$.
 The result (\ref{EN-z}) follows from the orthonormality of the Laplace-Beltrami eigenmodes in the $J_0$-weighted $L^2(\Gamma_b)$.
\end{proof}

We use this to establish the following bound.

\begin{prop}\label{PB-prop:bound}
Fix $\Gamma_b\in\cG_{k,\ell}^b$, then for $\eps>0$ sufficiently small there exists a constant $C>0$ such that
\beq
\label{e:tPiLL-bnd}
\| \tilde\Pi\LL_b v\|_{L^2(\Omega)} \leq C \eps \|\vv\|_{l^2},
\eeq
for all $\vv\in\RR^{N_{b}}$ with corresponding $v\in X_b$ given by (\ref{e:Xb-expand}).
\end{prop}
\begin{proof}

Any $v \in X_b$ has the form (\ref{e:Xb-expand}), which from (\ref{e:psibJ-exp}) admits the expansion
\beq
v = \eps^{-\frac12}\left( \ov_0(s)\hpsi_{b,0} + \ov_1(s)\hpsi_{b,1}\right)\left(1 - \frac12 \eps z K_1(s) +  O(\eps^2)\right).
\eeq
Turning to the expansion (\ref{PB-mbbL}) of $\LL_b$, the action of $\LL_b$ on $v$ has leading order term $\cL_b^2(\ov_0\hpsi_{b,0} + \ov_1\hpsi_{b,1})\in X_b$ which lies in the kernel of $\tilde\Pi$, and only higher order terms, in which at least one $\partial_z$ or $\nabla_s$ derivative falls upon $\tJ$ remain. Consequently we may write
\beq \label{e:Pi0cLv}
\tilde\Pi_0\cL_b^2 v = \eps^{\frac12}\left(Q_0\ov_0+Q_1\ov_1\right),
\eeq
where for $j=0, 1$ the differential operators $Q_j$ admit an expansion of the form
 \beq \label{e:Qj-form}
 Q_j:= Q_{j,0} + \eps Q_{j,1}+\eps^2 Q_{j,2} +\eps^3 Q_{j,3},
 \eeq
and each $Q_{j,k}$ is a $k$'th order differential operator in $\nabla_s$ with coefficients
that are smooth and decay exponentially in $z$ and have uniform $L^\infty(\Gamma_b)$  bounds in $s$ that are independent of
$\eps>0$ sufficiently small. Taking the $L^2(\Omega)$ norm of (\ref{e:Pi0cLv}), we transform to the local variables and integrate out
the $z$ dependence, the result is an expression of the form
\beq
 \|\tilde\Pi\cL_b^2 v \|_{L^2(\Omega)}^2 =  \eps^2 \int_{\Gamma_b}\left( \oQ_0\ov_0+\oQ_1\ov_1\right)^2\, J_0\, ds,
 \eeq
 where the $\oQ_j$ are third order differential operators of the form (\ref{e:Qj-form}) with $L^\infty(\Gamma_b)$ coefficients that only depend
 upon $s$.
In particular we deduce that
\beq
\| \eps^{k} \oQ_{j,k}\ov_j \|_{L^2(\Gamma_b)} \leq \|Q_{j,k}(1-\Delta_s)^{-k/2}\|_{L_*^2(\Omega)} \|\eps^{k} (1-\Delta_s)^{k/2} \ov \|_{L^2(\Gamma_b)},
\eeq
where the norm $\|\cdot\|_{L^2_*(\Omega)}$ denotes the induced $L^2\mapsto L^2$ operator norm on $L^2(\Gamma_b)$. From
classic elliptic regularity theory, the operator norm may be bounded, independent of $\eps>0$ sufficiently small,
while Lemma\,\ref{l-MCor} implies the existence of $C>0$ such that
\beq
\eps^{k}\|(1-\Delta_s)^{k/2}\ov\|_{L^2(\Gamma_b)} \leq C \|\vv\|_{l^2},
\eeq
for $k=0,1, 2, 3.$
In particular we deduce the existence of $C>0$, chosen independent of $\eps>0$ sufficiently small that
\beq
\label{e:cL2-bnd}
 \|\tilde{\Pi} \cL^2_b v \|_{L^2(\Omega)} \leq C \eps \|\vv\|_{l^2}.
\eeq

To extend this bound to  the full operator, we fix~$\lambda_*\in\rho(\cL_b^2)$, the resolvent set of $\cL_b^2$, and
use the expansion (\ref{PB-mbbL}) to write $\LL_b$ as
\beq
    \LL_b = (\cL_b^2-\lambda_*)+\eps\tilde\LL_b(\cL_b^2-\lambda_*)^{-1}(\cL_b^2-\lambda_*)+\lambda_*,
\eeq
where $\tilde\LL_b:=\LL_b-\cL_b^2$ denotes the lower order, relatively compact terms in $\LL_b.$
Act $\LL_b$ on $v\in X_b$ and project with $\tilde\Pi$, recalling that $\tilde{\Pi}\Pi =0$, we have the estimate
\beq
    \norm{\tilde\Pi\LL_b\Pi v}_{L^2(\Omega)}   \leq \norm{\tilde\Pi\cL_b^2v}_{L^2(\Omega)}+\eps\norm{\tilde\Pi\tilde\LL_b(\cL_b^2-\lambda_*)^{-1}(\Pi+\tilde\Pi)(\cL_b^2-\lambda_*)v}_{L^2(\Omega)}.
\eeq
Since the perturbation $\tilde{\LL}_b$ is relatively compact with respect to $\cL_b^2$ it follows that the composition
$\tilde{\LL}_b(\cL_b^2-\lambda_*)^{-1}$ is bounded uniformly in the $L^2(\Omega)$ operator norm, and from (\ref{e:cL2-bnd}) we
have an $O(\eps)$ contribution from the $\tilde{\Pi}\cL_b^2$ terms while a uniform bound on $\Pi\cL_b^2$ acting on
$X_b$ follows from Proposition\,\ref{PB-TheoCor} and the scaling (\ref{e:M00-asymp}) of the diagonal elements of the matrix $M^0$.
The result, (\ref{e:tPiLL-bnd}) follows.
\end{proof}

\subsection{Localization of the geometric and pearling eigenvalues of $\LL_b$}\label{PB-sec:Spectrum}

The estimates from the previous section permit the localization of the small eigenvalues of $\LL_b$ as perturbations of
pearling and geometric eigenvalues of $M$. While the off-diagonal terms $B$ and $B^T$ have $O(\eps)$ operator
bounds, their contributions to the small spectrum of $\LL_b$ is muted by the fact that
the restricted operator $\LL^\perp$ is uniformly $L^2(\Omega)$ coercive on $X_b^\bot$, satisfying $\sigma(\LL^\perp)\subset (C\eps^r, \infty)$
for some $C>0$, where $\eps^r$ is the bound in the definition of $\Sigma_{b,0}(r)$ and $\Sigma_{b,1}(r)$, see~\cite{HAYRAPETYAN2014SPECTRA} for details.

To localize the pearling eigenvalues of $\LL_b$ we consider $\Lambda\in\sigma(\LL_b)\cap(-\infty, C\eps^r)\subset \rho(\LL^\perp)$ for which the
resolvent operator $R(\Lambda,\LL^\perp)=(\Lambda-\LL^\perp)^{-1}$ is boundedly invertible.
We project the eigenvalue problem (\ref{e:EVP}) with $\Pi$ and $\tilde{\Pi}$ and decompose the eigenfunction $\Psi_b$
according to (\ref{e:EVP-decomp}). Using the invertability of $\LL^\perp-\Lambda$  we solve for $\Psi^\perp$, reducing
the eigenvalue problem for the system (\ref{PB-BreakL}) to
the equivalent finite dimensional system for $v_1$,
\beq\label{PB-eq:reducedEq}
   (M-\Lambda)v_1 = B(\LL^\perp-\Lambda)^{-1}B^T v_1.
\eeq

 We take the $l^2$-norm of both sides of~(\ref{PB-eq:reducedEq}),
and from the norm estimate (\ref{e:cL2-bnd}) on $B$ and $B^T$ we obtain
\beq\label{PB-eq:boundl2norm}
    \|(M-\Lambda)v_1 \|_{l^2}\leq c\eps^2 \norm{R(\Lambda,\LL^\perp)}_{L^2(\Omega)}\|v_1\|_{l^2}.
\eeq
Since $\LL^\perp$ is self-adjoint, we have the standard estimate
\beq\label{PB-eq:boundR}
    \norm{R(\Lambda;\LL^\perp)}_{L^2(\Omega)}\leq \left(\text{dist}(\Lambda,\sigma(\LL^\perp))\right)^{-1}\leq\frac{1}{|\Lambda-C\eps^r|}.
\eeq
However, $M$ is also self-adjoint and we have the lower bound,
\beq\label{PB-eq:DistBound}
    \|(M-\Lambda)v_1 \|_{l^2} \geq \textrm{dist}(\sigma(M),\Lambda)\|v_1 \|_{l^2},
\eeq
and combining the upper and lower bounds yields the localization
\beq
\label{e:M-cLb}
    \text{dist}(\sigma(M),\Lambda)\leq \frac{c\eps^2}{|\Lambda-C\eps^r|}\leq c\eps^{2-r},\quad \text{for } \Lambda<\frac{C}{2}\eps^r.
\eeq
In particular the difference between the pearling spectrum of $\LL_b$ and the eigenvalues of $M$ is smaller than their generically
$O(\eps)$ size. The optimal control on the eigenvalues of $\LL_b$ is achieved when the $2-r$ exponent in (\ref{e:M-cLb}) balances with the $\gamma$ given
in Proposition\,\ref{PB-TheoCor}, which occurs for the choice $r=r^*_b=4\frac{d-1}{d+8}.$ Moreover the converse inclusion also holds,
 if $u$ is an eigenfunction of $M$ corresponding to an eigenvalue $\lambda_0<\delta$, then by classical regular perturbation results there is
 a nearby $v_1=u+O(\eps^2)$ and $\lambda=\lambda_0+O(\eps^2)$ which satisfies (\ref{PB-eq:reducedEq}). We have established the following result.

\begin{theo}\label{PB-PearlingCond}
For space dimension $d=2, 3$, fix a class $\cG_{K,\ell}^b$ of admissible codimension-one interfaces, and let $\LL_b$ be the second variation of
$\cF$ about the associated bilayer morphology $u_b(\cdot; \Gamma_b, \hlam)$ given by (\ref{e:ub-exp}). Fix $r=r_b^*:=4 \frac{d-1}{d+8}$ in the
definition (\ref{e:SigmaBl}) of $\Sigma_{b,0}(r)$ and $\Sigma_{b,1}(r)$.  Then there exists $C>0$ such that for all $\eps>0$ sufficiently small,
the set $\sigma(\LL_b)\cap (-\infty, C\eps^{r_b^*})$ consists of the union of pearling eigenvalues
$ \bigl\{\Lambda_{b;0,k}\,\bigl |\, k\in\Sigma_{b,0}(r_b^*)\bigr\}$
and the geometric eigenvalues $\bigl\{\Lambda_{b;1,k}\,\bigl | \, k\in\Sigma_{b,1}(r_b^*)\bigr\}$,
which satisfy the asymptotic expansions
\begin{align}
\label{e:Pevs-Bl}
	\Lambda_{b;0,k} &= \eps\left(\frac{P_{k,0}^2-\hlam S_b}{\|\hpsi_{b,0}\|_{L^2(\RR)}^2}-\eta_d\lambda_{b,0}\right)+O\!\left(\eps^{2-r_b^*}\right), & {\rm for\,} k\in\Sigma_{b,0}(r_b^*),\\
\label{e:Gevs-Bl}
		\Lambda_{b;1,k} &= \eps \frac{P_{k,1}^2}{\|\spsi_{b,1}\|_{L^2(\RR)}^2} + O\!\left(\eps^{2-r_b^*}\right), & {\rm for}\, k\in\Sigma_{b,1}(r_b^*),
\end{align}
where all quantities are as defined in Proposition\,\ref{p:M0}.  Moreover,
the associated codimension-one bilayer interface is stable with respect to the pearling eigenvalues, if
\beq\label{BP-eq:PearlingStableCond}
	\hlam S_b+\eta_d\lambda_{b,0}\|\hpsi_{b,0}\|_{L^2(\RR)}^2<0,
\eeq
and is pearling unstable if this quantity is positive.
\end{theo}

\section{Pearling eigenvalues of filament morphologies}\label{Sec:Pearling-Pr}

In this section we localize the pearling eigenvalues associated to filament morphologies obtained from the dressing  of admissible
codimension-two hypersurfaces embedded in $\RR^3$. Fixing an admissible codimension-two hypersurface, $\Gamma_\fil$, we study the second variational derivative, $\LL$, of $\cF$ evaluated at the associated filament morphology $u_\fil$  defined in (\ref{e:PR}). While the general form of $\LL$ is given by (\ref{e:LL}),
when the linearization acts on functions $u\in H^4(\Omega)$ whose support lies in the reach, $\Gamma_\fil^\ell$ of $\Gamma_\fil$, we may exploit the notation of section \ref{sec:coord-pr} to express  $\LL$ in the equivalent form
\beq
    \LL_\fil = (L_{\fil}^{u_\fil}-\eps D_z+\eps^2\partial_G^2)^2+\eps\left[\eta_1(\Delta_z-\eps D_z+\eps^2\partial_G^2)-\eta_2W''(u_\fil)\right]-\left(\Delta_zu_\fil-W'(u_\fil)+\eps D_zu_\fil\right)W'''(u_\fil),
\eeq
where~$L_{\fil}^{u_\fil}$ was defined in~(\ref{e:Lf-uf-def}). From the definition (\ref{e:partialG2}) of $\partial_G^2$ and the form (\ref{e:PrJ}) of
$\tilde{J}_f$ we arrive at the expansion,
\beq
\label{e:pG-expand}
\partial_G^2 = \partial_s^2+\eps \vz\cdot(2\vk\partial_s^2+\partial_s\vk\partial_s) + O(\eps^2).
\eeq
The two dominant operators in $\LL_\fil$ are~$L_{\fil}$, defined in (\ref{e:Lf-def}), and $\eps^2\partial_s^2$, which motivates the introduction of
\beq\label{e:PrFullL}
    \cL_\fil:=L_{\fil}+\eps^2\partial_s^2,
\eeq
which balances the sum of the mostly negative self-adjoint operator $L_{\fil}$, and the non-positive line diffusion
operator associated to $\Gamma_\fil$. With this notation, we may write $\LL_\fil$ as the square of $\cL_\fil$ plus lower order terms
\begin{align}
\label{LL-def}
    \LL_\fil =& \left(\cL_{\fil}-\eps\left(D_z+W'''(U_\fil)U_{\fil,1})\right)\right)^2\\
    &+\eps\left[\eta_1(\Delta_z-\eps D_z+\eps^2\partial_G^2)-\eta_2W''(u_\fil)\right]-\left(\Delta_zu_\fil-W'(u_\fil)+\eps D_zu_\fil\right)W'''(u_\fil)+O(\eps^2),\nonumber
\end{align}
and note that the error terms in~$\LL_\fil$ are relatively compact and uniformly bounded with respect to~$\cL^2_\fil$.
In this section we analyze the spectrum of~$\cL_\fil$ on spaces where the positive spectra of $L_{\fil}$ balance the negative spectra of $\eps^2\partial_s^2$.

\begin{definition}\label{cs-eq:ScaledEigenvaluesPr}
We define the \textbf{scaled eigenfunctions}~$\spsi_{\fil,k}:=\chi(z) J_\fil^{-1/2}\hpsi_{\fil,k}$, where~$\hpsi_{\fil,k}$ is the~$k^{th}$ eigenfunction of~$L_{\fil}$ and~$\chi(z)$ is a~$C^\infty$ cut-off function.
\end{definition}


To analyze the spectrum of~$\LL_\fil$ we follow \cite{DAI2015COMPETITIVE}, expand the Cartesian $z$-Laplacian in polar form
\beq
    L_\fil:=\Delta_z-W''(\phi_\fil)=\partial_R^2+\frac{1}{R}\partial_R+\frac{1}{R^2}\partial_\theta^2-W''(\phi_\fil),
\eeq
and introduce the spaces~$\cZ_m$, defined by
\beq\label{CS-eq:ZmSpace}
    \cZ_m:=\{f(R)\cos(m\theta)+g(R)\sin(m\theta)~\big|~ f,g\in C_c^\infty(0,\infty),m\in\NN\}.
\eeq
These spaces are invariant under  $L_\fil$, and mutually orthogonal in~$L^2(\Omega)$. Moreover, on these spaces the action of
$L_\fil$ reduces to
\beq
    L_\fil(f(R)\cos(m\theta)+g(R)\sin(m\theta)) = \cos(m\theta)L_{\fil,m}f+\sin(m\theta)L_{\fil,m}g,
\eeq
where, consistent with the notation in (\ref{e:Lf0-def}), we have introduced
\beq\label{CS-eq:Lm}
    L_{\fil,m}:=\frac{\partial^2}{\partial R^2}+\frac{1}{R}\frac{\partial }{\partial R}-\frac{m^2}{R^2}-W''(U_\fil).
\eeq
Each operator~$L_{\fil,m}$ is self-adjoint in the~$R$-weighted inner product,
\beq
 (f,g)_R:=\int_0^\infty f(R)g(R)R\,dR,
 \eeq
and the operator~$L_{\fil,1}$ has a 1-dimensional kernel spanned by its ground state $\partial_R \phi_\fil>0$. We deduce that
$L_{\fil,1}\leq0$, and since $(L_{\fil,m}f,f)_R<(L_{\fil,1}f,f)_R$ for $m>1$, it follows $L_{\fil,m}<0$ and in particular
$L_{\fil,m}$ is boundedly invertible for all $m>1$. Conversely, $(L_{\fil,0}f,f)_R>(L_{\fil,1}f,f)_R$ and in particular
 $(L_{\fil,0}\partial_R\phi_{\fil},\partial_R\phi_{\fil})_R>0$ so that $L_{\fil,0}$ must have a nontrivial positive subspace.
We  denote the eigenfunctions and eigenvalues of $L_{\fil,m}$ by $\{\spsi_{\fil,m,j}\}_{j=0}^\infty$ and $\{\lambda_{\fil,m,j}\}_{j=0}^\infty$, respectively,
and drop the subscript $\fil$ when doing so does not cause confusion.
It follows from equation~(\ref{e:FLcp}) that
$$ {\rm span}\{\partial_{z_1}U_\fil,\partial_{z_2}U_\fil\}= {\rm span}\{\phi_\fil^\prime(R) \cos\theta,\phi_\fil^\prime(R) \sin\theta\}\subset \cZ_1\cap \ker L_{f}.$$
The following assumption guarantees that the kernel of $L_f$ is indeed two dimensional,  and that the operators $L_{\fil,j}$ are strictly
negative for $j\ge 1$; in particular, $\lambda_{0,0}>0,$  $\lambda_{0,1}=0$, and $\lambda_{0,j}<0$  for every~$j\ge1$.

\begin{assum}\label{CS-Assum:Kernel}
The operator~$L_{\fil,0}$ has no kernel and a one-dimensional positive eigenspace, the operator $L_{\fil, 1}$ is negative except for
a one dimensional kernel spanned by $\{ \phi_{\fil}^\prime\}$.
\end{assum}

We consider the eigenvalue problem
\beq\label{e:EVP-pr}
    \LL_\fil\Psi_{\fil} = \Lambda_{\fil}\Psi_{\fil},
\eeq
associated to the second variation,~$\LL_\fil$, of~$\cF$ about a filament morphology $u_\fil$, given by (\ref{LL-def}).
We show that there exists constant $U_\fil>0$, independent of $\eps>0$, such that the eigenfunctions associated to $\LL_\fil$ corresponding
to eigenvalues $\Lambda_{\fil}<U_\fil$, comprise two sets: the \emph{pearling eigenmodes}, enumerated as $\{\Psi_{\fil;0,n}\}_{n=N_1}^{N_2}$ and the \emph{meander eigenmodes}, enumerated as $\{\Psi_{\fil;1,n}\}_{n=0}^{N_3}$.  On the filament $\Gamma_\fil$, where $z=0$, the codimension-two
Laplacian $\partial_G^2$, introduced in (\ref{e:partialG2}), reduces to the line-diffusion operator $\partial_s^2,$ where $s$ denotes arc-length along $\Gamma_\fil.$ We denote the corresponding eigenfunctions of $-\partial_s^2$ by $\{\Theta_{\fil,n}\}_{n=0}^\infty$ with eigenvalues $\beta_{\fil,n}\geq0,$
and introduce the sets
\beq\label{e:SigmaPr}
\Sigma_{\fil,j}(r_f):= \left\{n\in\NN_+\,\bigl|~(\lambda_{\fil,j}-\eps^2\beta_{\fil,n})^2\leq \eps^{r_\fil} \right\},\qquad j=0,1,
\eeq
for $r_\fil\in(0,1)$. The line-diffusion eigenvalues grow like $\beta_{\fil,n}\tilde n^2$, and the sizes $N_{\fil,j}:=|\Sigma_{\fil,j}(r_\fil)|$
satisfies the same asymptotic relations as $N_{b,j}$ given in (\ref{Nj0-asymp}), for $d=3$ and $r$ replaced with $r_\fil.$

The spectrum of $\LL_\fil$ cannot be localized by a regular perturbation expansion since the eigenvalues are asymptotically
close together. A perturbation analysis requires bounds on the spectrum that are uniform in $\eps\ll 1$. We establish the tensor product
formulation of the pearling and meander eigenmodes,
\beq
 \Psi_{\fil,j,n} = \spsi_{\fil,j}(z)\Theta_{\fil,n}(s) + O(\eps),
 \eeq
and introduce the~$L^2(\Omega)$ orthogonal projection~$\Pi$ onto the space
\beq\label{e:Xfil}
    X_\fil(\Gamma_\fil):=\text{span}\{\spsi_{\fil,j}(z)\Theta_n(s)~|~j=0,1,~\text{and } n\in\Sigma_{\fil,j}(r_\fil)~\text{respectively}\},
\eeq
which approximates the eigenspaces of $\LL_\fil$ corresponding to pearling and meander eigenmodes.
 The space $X_\fil$ has dimension $N_\fil=N_{\fil,0}+N_{\fil,1},$ see (\ref{e:SigmaPr}). In particular we remark that basis of $X_\fil$ is orthonormal in $L^2(\Omega)$, up to exponentially small terms. Indeed, dropping the $\fil$ subscript,
the basis elements are localized on $\Gamma_\fil$ and their inner product can be written as
\beq
\left(\spsi_k\Theta_n, \spsi_j\Theta_m\right)_{L^2(\Omega)}=\int_{\Gamma_\fil}\int_{0}^\frac{\ell}{\eps} \spsi_j(z,s)\spsi_k(z,s)\Theta_n(s)\Theta_m(s)J_\fil(z,s)dz\,ds= \delta_{nm}\delta_{jk}
                                                                                +O(e^{-\ell/\eps}),
\eeq
where the second equality follows from Definition \ref{cs-eq:ScaledEigenvaluesPr}, of the scaled eigenfunctions, and the factored form (\ref{e:PrJ}) of the
Jacobian. We change to whiskered coordinates in the Laplacian according to (\ref{e:PrClosedLap}), and use the form (\ref{e:PR}) of $u_\fil$ to expand
the operator $\LL_\fil$ as
\beq\label{e:PP-mbbL}
\bL_{\fil}= \cL_\fil^2+\eps\LL_{\fil,1}+\eps^2\LL_{\fil,2}+O(\eps^3),
\eeq
where~$\cL_\fil$ is defined in~(\ref{e:PrFullL}), and
\beq
    \LL_{\fil,1} := -\cL_\fil\circ (D_z+W'''(U_\fil)U_{\fil,1})-(D_z+W'''(U_\fil)U_{\fil,1}+\eta_1)\circ\cL_\fil+\eta_dW''(U_\fil)-(D_zU_\fil-\cL_\fil U_{\fil,1})W'''(U_\fil).\label{e:cLfil1}
\eeq
For~$i\geq1$, the operators $\LL_{\fil,i}$ are compact relative to $\cL_\fil$ in~$H_0^1(\Gamma_\fil^\ell)$.

We localize the pearling and geometric eigenvalues of $\LL_\fil$  via an analysis of its projection onto $X_\fil$,
searching for solutions of the eigenvalue problem,~(\ref{e:EVP-pr}), via the decomposition
\beq
\label{e:PR-EVP-decomp}
    \Psi_\fil = v_\fil+ \Psi_\fil^\perp,
\eeq
where
\beq
\label{e:Pr-v-alpha}
v_\fil=\sum_{k\in\Sigma_{\fil,0}(r_\fil)}\alpha_{0,k}\spsi_{\fil,0}\Theta_k+\sum_{k\in\Sigma_{\fil,1}(r_\fil)}\alpha_{1,k}\spsi_{\fil,1}\Theta_k \in X_\fil
\eeq
and  $\Psi_{\fil}^\perp\in X_\fil^\bot$.  We use the projections $\Pi$ and $\tilde{\Pi}=I-\Pi$ to decompose the operator~$\LL_\fil$ into the $2\times2$ block form,
\beq\label{PP-BreakL}
    \begin{bmatrix}
        M & B           \\[0.3em]
        B^T & \LL^\perp
      \end{bmatrix},
\eeq
where
\begin{align}
\label{def-BreakL-fil}
    M:=\Pi\LL_\fil\Pi,\quad
    B:=\Pi\LL_\fil\tilde{\Pi},\quad
    \LL^\perp:=\tilde \Pi\LL_\fil\tilde \Pi.
\end{align}
The decomposition hinges upon the uniform $L^2$ coercivity of the operator on $\LL_\fil$ on the space $X_\fil^\bot$.
This result was established for bilayers in \cite{HAYRAPETYAN2014SPECTRA}, and is assumed here.

\begin{assum}\label{CS-Assum:coercive}
Fix $K, \ell>0$. There exists $0<r_{\fil,0}<1$ and $\rho>0$ such that for all $r_\fil\in[0,r_{\fil,0})$,  and all admissible $\Gamma_f\in\cG^\fil_{K,\ell}$, we have
\beq
 \left( \LL_\fil w, w\right)_{L^2(\Omega)} \geq \rho \eps^{r_\fil} \|w\|_{L^2(\Omega)}^2, \label{LLfil-CA}
 \eeq
 for all $w\in X_{\fil}^\bot(\Gamma_\fil, r_\fil).$
\end{assum}

In section~\ref{PP-sec:BoundM} we analyze the spectrum of $M$, and show that its spectrum coincides to
leading order with the small eigenvalues of $\LL_\fil$.

\subsection{Eigenvalues of the pearling matrix~$M:=\Pi\LL_\fil\Pi$}\label{PP-sec:BoundM}
Denote~$v\in X_\fil$ by
\beq
\label{e:Xfil-expand}
v(s,z)=\overline{v}_0(s)\spsi_{\fil,0}+\ov_1(s)\spsi_{\fil,1},
\eeq
where for $j=0, 1$ we have introduced
\beq
\ov_j:=\sum_{k\in\Sigma_{\fil,j}(r_\fil)}\alpha_{j,k}\Theta_k.
\eeq
Representing the coefficients of $v$ by  $\valpha_j=(\alpha_{j,k})_{k\in\Sigma_{\fil,j}(r_\fil)}$, for $j=0, 1$ and $\valpha=(\valpha_0,\valpha_1)^T$,
 the action of $\LL_\fil$  on $v$ can be represented by
the matrix $M\in\mbbR^{N_\fil\times N_\fil}$  with entries
\beq\label{PP-Def:M}
    M_{jk}:=(\LL_\fil\spsi_{\fil,I(j)}\Theta_j,\spsi_{\fil,I(k)}\Theta_k)_{L^2(\Omega)},
\eeq
where the index function $I(j)$ takes the value $k$ if $j\in\Sigma_{\fil,k}(r_\fil)$ for $k=0, 1$.
From the expansion of~$\LL_\fil$, (\ref{e:PP-mbbL}), we may fix $q\in\NN_+$ and
group terms in $M$ into two classes
\beq\label{PP-def:M}
    M= M^0 +\eps^q\tilde M,
\eeq
where
\begin{align}
    \label{PP-M0}M^0_{jk}&=(\cL_\fil^2\spsi_{\fil,I(j)}\Theta_j,\spsi_{\fil,I(k)}\Theta_k)_{L^2(\Omega)}+
     \sum_{i=1}^q\eps^i(\LL_i\spsi_{\fil,I(j)}\Theta_j,\spsi_{\fil,I(k)}\Theta_k)_{L^2(\Omega)},\\
    \tilde M_{jk}&=\sum_{i\geq q}\eps^{(i-q)}(\LL_i\spsi_{\fil,I(j)}\Theta_j,\spsi_{\fil,I(k)}\Theta_k)_{L^2(\Omega)}.
\end{align}
For $N_\fil:=N_{\fil,0}+N_{\fil,1}$, Lemma~\ref{l-MCor} applied with $d= 3$, $r_\fil>0$, and $q=2$, implies that an $O(1)$ bound on the
$l^\infty(N_\fil\times N_\fil)$ norm of $\tilde{M}$ implies a $o(\eps)$ bound on the $l^2(N_\fil)\mapsto l^2(N_\fil)$ operator norm, denoted $l^2_*$, of $\eps^q\tilde{M}.$
We further divide $M^0$ into sub-blocks
\beq\label{PP-M0-sub}
M^0 =    \begin{bmatrix}
        M^{0,0} & M^{0,1}           \\[0.3em]
        M^{1,0} & M^{1,1}
      \end{bmatrix},
\eeq
with the $M^{j,k}$ sub-block corresponding the inner products with entries from $\Sigma_{\fil,j}(r_\fil)$ and $\Sigma_{\fil,k}(r_\fil)$ for $j,k=0,1.$
The $M^{0,0}$ and $M^{1,1}$ sub-blocks are called the pearling and geometric sub-blocks, respectively. The following proposition characterizes the leading order entries of the sub-blocks in terms of system dependent parameters
\begin{prop}
\label{p:M0pr}
For an admissible  filament morphology $u_\fil(\cdot; \Gamma_\fil, \hlam)$ with far-field parameter $\hlam$,  the entries of the pearling, geometric,
and cross-term sub-blocks of $M^0$ defined in (\ref{PP-M0-sub}) take the form
\beq \label{e:pr-M00-asymp}
    M^{0,0}_{jk}=
    \begin{cases}
        \eps\left(P_{f,k,0}^2-\hlam S_\fil-\eta_d{\left(\norm{\hpsi_{\fil,0}'}_{L_R}^2+\lambda_{\fil,0}\norm{\hpsi_{\fil,0}}_{L_R}^2\right)}\right) + O(\eps\sqrt{\eps}) &~\text{ if }j=k,\\
      -\eps^2 \int_{\Gamma_\fil}\left(\norm{\nabla_z\hpsi_{\fil,0}}_{L^R}^2+S_{\fil,1,0}\right)|\vkappa|^2\Theta_k\Theta_j\,ds
        +O(\eps^2\sqrt{\eps})&~\text{ if }j\neq k,
    \end{cases}
\eeq
for $j,k\in\Sigma_{\fil,0}(r_\fil),$ where $S_\fil$ and $S_{\fil,1,0}$ are given in (\ref{e:S-pr}) and (\ref{e:S1j-def-fil}) respectively,
\beq \label{e:pr-M11-asymp}
    M^{1,1}_{jk} =
    \begin{cases}
        \eps P_{k,1}^2 + O(\eps^2) &~\text{ if }j=k,\\
        -\eps^2\int_{\Gamma_\fil}\left(\norm{\nabla_z\hpsi_{\fil,1}}_{L^R}^2+ S_{\fil,1,1}\right)|\vkappa|^2\Theta_k\Theta_j\,ds +O(\eps^2\sqrt{\eps})&~\text{ if }j\neq k,
    \end{cases}
\eeq
for $j,k\in\Sigma_{\fil,1}(r_\fil),$ where $S_{\fil,1,1}$ is given in (\ref{e:S1j-def-fil}), and
\beq\label{e:pr-M01}
    M^{0,1}_{jk} =    -\eps S_{\fil,2} \cdot\int_{\Gamma_\fil} \vkappa \Theta_j\Theta_k \, ds +\cL_{\fil;jk}+ O(\eps^2),
\eeq
for $j\in\Sigma_{\fil,0}(r_\fil)$, $k\in\Sigma_{\fil,1}(r_\fil),$ where $S_{\fil,2}$ is given in (\ref{e:pr-S2j-def}) and $\cL_{\fil;jk}=O(\eps^{\frac{3}{2}})$ is given in (\ref{e:cLfil-asymp}). The vector curvature $\vkappa$ of $\Gamma_\fil$ is defined in (\ref{CS-eq:VecKappa}),
while the detuning constants,
\beq\label{PB-Pk}
    P_{\fil,k,I(k)} := \eps^{-1/2}(\lambda_{\fil,I(k)}-\eps^2\beta_k),
\eeq
arise in the definition of $\Sigma_{\fil,I(k)}(r_\fil)$ in (\ref{e:SigmaPr}).
\end{prop}
\begin{proof}
The scaled eigenfunctions have support within $\Gamma_\fil^\ell$ and hence we may change to the whiskered coordinates. Using the explicit form of the Jacobian,~(\ref{e:PrJ}), the self adjointedness of $\cL_\fil$  in the $L^2(\Omega)$ inner product, and the expansion of $\cL_\fil$, (\ref{e:PrFullL}),
\beq
 \cL_{\fil;jk}:=(\cL_\fil^2\spsi_{I(j)}\Theta_j,\spsi_{I(k)}\Theta_k)_{L^2(\Omega)} =
 \int_{\Gamma_\fil}\int_{0}^{\ell/\eps} \cL_\fil\left( \Frac{\hpsi_{I(j)}\Theta_j}{\sqrt{J_\fil}}\right) \cL_\fil\left( \hpsi_{I(k)} \Theta_k \sqrt{J_\fil}\right)\, dz\, ds,
 \eeq
 where here and below we have dropped the $\fil$ subscript on $\hpsi$ and $\Theta$ and their eigenvalues.
 We deduce from (\ref{e:PrJ}) that $J_\fil^p$ is $O(\eps^{2p})$ for any $p\in\NN_+$, while
 $\nabla_z J_\fil =\eps^3\vkappa$, $\partial_s J_\fil=\eps^3z\cdot\nabla_s\vkappa$, $\partial_s^2 J_\fil =\eps^3z\cdot\partial_s^2 \vkappa$
 and $\Delta_z J_\fil = 0$. Moreover
 \beq
 \cL_\fil \hpsi_{I(j)}\Theta_j = (\lambda_{I(j)}-\eps^2 \beta_j)\hpsi_{I(j)}\Theta_j = \eps^{\frac12} P_{\fil, j,I(j)} \hpsi_{I(j)}\Theta_j,
\eeq
so we identify the leading order terms
\begin{align}
 \cL_\fil\left( \Frac{\hpsi_{I(j)}\Theta_j}{\sqrt{\tJ}}\right) &= \eps^{\frac12}P_{\fil,j,I(j)}  \Frac{\hpsi_{I(j)}\Theta_j}{\sqrt{J_\fil}}
 -
 \nabla_z\hpsi_{I(j)}\cdot\vkappa\Theta_j
  +O(\eps^3),\label{e:cLJ-12}\\
 \cL_\fil\left( \hpsi_{I(k)}\Theta_k\sqrt{J_\fil}\right) &= \eps^{\frac12}P_{\fil,k,I(k)}  \hpsi_{I(k)}\Theta_k \sqrt{J_\fil} + \eps^2 \nabla_z\hpsi_{I(k)}\cdot\vkappa\Theta_k +
 O(\eps^{3}),\label{e:cLsqrtJ}
 \end{align}
 where we used the estimate $\|\partial_s \Theta_j\|_{L^2(\Gamma_\fil)} =O(\eps^{-1})$ to bound the $\partial_s \Theta_j\cdot \partial_s J_\fil$ terms in the error. Combining these,
and using the parity considerations to eliminate the $z$ integrals of $\hpsi \nabla_z \hpsi$ we find
\begin{align}\label{e:cLfil-asymp}
 \cL_{\fil;jk} &=
 \begin{cases}
  \eps P_{\fil,j,I}P_{\fil,k,I}\delta_{jk}  -\eps^2 \|\nabla_z\hpsi_{I}\|_{L^2}^2 \int_{\Gamma_\fil} |\vkappa|^2 \Theta_j\Theta_k\,ds + O(\eps^{\frac{7}{2}}) & \text{if}~I(j)=I(k)=I,\\
  \eps^{\frac{3}{2}}(P_{\fil,j,I(j)}+P_{\fil,k,I(k)})\int_0^\infty \hpsi_{I(j)}\nabla_z \hpsi_{I(k)}\,dz \cdot\int_{\Gamma_\fil}\vkappa\Theta_j\Theta_k\,ds +O(\eps^\frac52)& \text{if}~ I(j)\neq I(k).
  \end{cases}
  \end{align}
 The reduction of the $\LL_{\fil,1}$ inner products are considered case by case. For $j=k\in\Sigma_{\fil,0}(r_\fil)$ we only require leading-order terms. Equations (\ref{e:cLJ-12}) and (\ref{e:cLsqrtJ}) imply that the first two terms in the inner product $(\LL_{\fil,1}\spsi_{I(j)}\Theta_j,\spsi_{I(k)}\Theta_k)_{L^2(\Omega)}$ are of lower order, and the leading order is given by
\begin{align}
\left(\LL_{\fil,1}\spsi_{0}\Theta_j,\spsi_{0}\Theta_k\right)_{L^2(\Omega)} &=
        -\Bigl(\bigl((L_\fil U_{\fil,1})W'''(\phi_\fil)-\eta_dW''(\phi_\fil)\bigr)\spsi_{0}\Theta_j,\spsi_{0}\Theta_k\Bigr)_{L^2(\Omega)}+O(\sqrt\eps)\\
        &
        =-\hlam S_\fil-\eta_d\left(\norm{(\hpsi_{\fil,0})'}_{L_R}^2+\lambda_{\fil,0}\norm{{\hpsi}_{\fil,0}}_{L_R}^2\right)+O(\sqrt\eps),\label{e:PrL1Terms}
\end{align}
where we have introduced the 'background factor'
 \beq\label{e:S-pr}
    S_\fil:=2\pi\int_0^\infty\Phi_{\fil,1}W'''(\phi_\fil)\hpsi_{\fil,0}^2\,RdR,
\eeq
which characterizes the impact of the the far-field value of $u_\fil$ on the pearling eigenvalues.
Equation~(\ref{e:PrL1Terms}), in conjunction with (\ref{e:cLfil-asymp}) yields the $j=k$ case of (\ref{e:pr-M00-asymp}).

 For $j\neq k$, the $\eps^2\partial_s^2$ term in $\LL_\fil$ induce lower order contributions to the inner product.
This is clear, unless the term falls entirely upon $\Theta_i$ or upon $\Theta_j$,
in which case it becomes $\eps^2\beta_i$ or $\eps^2\beta_j$ (which are $O(1)$) but the integral is lower
order because of orthonormality. The term $\eps^2 \partial_s\Theta_i\partial_s \Theta_j$ might formally appear to be leading order, yet an integration by parts returns us to the prior case.
Therefore, the inner product takes the form
 \beq
  \left( \LL_{\fil,1} \spsi_{\fil,I(j)}\Theta_j, \spsi_{\fil,I(k)}\Theta_k\right)_{L^2(\Omega)} = \int_{\Gamma_\fil} \left( \LL_{\fil,1} \frac{\hpsi_{\fil,I(j)}}{\sqrt{J_\fil}},  \frac{\hpsi_{\fil,I(k)}}{\sqrt{J_\fil}}\right)_{\tJ_\fil} \Theta_j\Theta_k \, ds,
 \eeq
which is zero unless the $J_\fil$ inner product has non-trivial $s$ dependence. The only leading order term in $\LL_{\fil,1}$ with non-trivial
$s$ dependence is $-D_zU_\fil=-\vkappa \cdot\nabla_zU_\fil + O(\eps)$, however $\nabla_zU_\fil$ is odd in $z$ so parity issues yield
a non-zero $z$-integral only if $I(j)\neq I(k)$, for which we find
\beq \label{e:M01-est-fil}
 \left( \LL_{\fil,1} \spsi_{\fil,I(j)},  \spsi_{\fil,I(k)}\right)_{J_\fil} =  S_{\fil,2} \vkappa(s)  
 + O(\eps),
\eeq
where we have introduced the scalar constant
\begin{align}
    S_{\fil,2} =&2\pi \int_0^\infty W'''(\phi_\fil){\hpsi_{\fil,0}}{\hpsi_{\fil,1}}\partial_R \phi_\fil \, RdR>0.
    \label{e:pr-S2j-def}
\end{align}

For $I(j)=I(k)$ and~$j\neq k$ the leading order terms occur at next order, and we seek terms which introduce $s$ dependence whose combination
preserves even parity in $z$.  The operator $\LL_{\fil,1}$ can be decomposed into ``even'' operators which preserve $z$ parity and ``odd'' operators
which map odd parity functions to even ones, and conversely. The odd component is the single term
$$[\LL_{\fil,1}]_{\rm odd}=\vkappa\cdot\nabla_z\phi_\fil W'''(\phi_\fil)+O(\eps^2),$$
while its even component can be further separated into terms, $\LL_{\fil,1}^0$ with no $s$ dependence and a single term
$$[\LL_{\fil,1}]_{\rm even}=[\LL_{\fil,1}^0]_{\rm even}+ [D_z\phi_\fil W'''(\phi_\fil)]_{\rm even}= [\LL_{\fil,1}^0]_{\rm even}+\eps (\vz\cdot\vkappa) \kappa\cdot\nabla_zU_\fil W'''(U_\fil) +O(\eps^3),$$
Viewing the integrand in $ \left( \LL_{\fil,1} \spsi_{\fil,I(j)},  \spsi_{\fil,I(k)}\right)_{J_\fil} $ as the action of $\LL_{\fil,1}$ on $\spsi_{\fil,I(j)}$ subsequently multiplied by $\spsi_{\fil,I(k)}$ and $J_\fil$,  we recall (\ref{e:PrJ}) and expand
\beq \label{e:psiJ-exp-fil}
\spsi_{\fil,I(j)}=\hpsi_{\fil,I(j)}J_\fil^{-\frac12} = \hpsi_{b,I(j)}\eps^{-1}\left(1 + \frac12 \eps z\cdot\vkappa +  O(\eps^2)\right).
\eeq
At leading order we obtain
\beq\label{e:LL1-even-fil}
  \left([\LL_{\fil,1}]_{\rm even} \spsi_{\fil,I(j)},  \spsi_{\fil,I(k)}\right)_\tJ = \eps S_{\fil,1,j} |\vkappa|^2 + O(\eps^2),
\eeq
where we have introduced the system dependent quantity
\beq
\label{e:S1j-def-fil}
    S_{\fil,1,j}:=\int_{\RR^2} z\cdot\nabla_z(W''(\phi_\fil))\hpsi_{\fil,I(j)}^2\,dz= 2\pi \int_0^\infty \partial_R (W''(\phi_\fil)) \psi_{\fil,I(j)}^2R^2\,dR,
 \eeq
 for $j=0,1.$
 For the odd part of $\LL_1$, overall even $z$ parity requires one odd term from $\spsi_{b,I(j)}$, $\spsi_{b,I(k)}$ or $J_\fil$, consequently
\beq \label{e:LL1-odd-fil}
  \left([\LL_{\fil,1}]_{\rm odd} \spsi_{\fil,I(j)},  \spsi_{\fil,I(k)}\right)_{J_\fil} =O(\eps^2),
\eeq
and combining (\ref{e:LL1-even-fil}) and (\ref{e:LL1-odd-fil}) yields the $j\neq k$ cases of both (\ref{e:pr-M00-asymp}) and (\ref{e:pr-M11-asymp}).
\end{proof}

The diagonal entries, $M^0_{\text{diag}}$, of $M^0$ are $O(\eps)$, and we wish to show that the matrix of off-diagonal entries,
$M^0_{\text{off-diag}}:=M^0-M^0_{\rm diag}$ has an $o(\eps)$ $l^2_*$ norm, so that the eigenvalues of
$M^0$ correspond with its diagonal entries to leading order. From Lemma \ref{l-MCor} we have appropriate operator norm bounds for all
off-diagonal entries of $O(\eps^2\sqrt{\eps})$. The remainder of the off-diagonal entries can be combined into matrices with $\Theta_j\Theta_k$
factored entries that enjoy bounds independent of the size of the space $X_\fil.$

\begin{lemma}\label{l:prM0}
    Let~$\Gamma_\fil\in\cG^\fil_{K,\ell}$ be an admissible hypersurface, then its curvatures $\vkappa\in W^{2,\infty}(S)$. Let~$f:\RR^2\rightarrow\RR$ be a smooth function, and then there exists $C>0$ such that for any pair of index sets $\Sigma_j\subset\NN_+$ with $|\Sigma_j|=N_j$ for $j=1,2$,
  the matrix~$A\in\RR^{N_1\times N_2}$, defined by its entries
    \beq
        A_{ij} = \int_{\Gamma_\fil} f(\vkappa\,)\Theta_i\Theta_j\,ds,
    \eeq
    where $\{\Theta_k\}_{k=0}^\infty$ are the eigenfunctions of line diffusion operator associated to $\Gamma_\fil$; satisfies
    \beq
        \norm{A}_{\RR^{N_2}\mapsto \RR^{N_1}}\leq C.
    \eeq
 If moreover $\Sigma_1=\Sigma_{\fil,0}(r_\fil)$ and $\Sigma_2=\Sigma_{\fil,1}(r_\fil)$,
 then we have the estimate
 \beq
  \norm{A}_{\RR^{N_2} \mapsto \RR^{N_1}}\leq C\eps^s,
 \eeq
 where $s:=\frac{3}{8}r_\fil>0.$
\end{lemma}
\begin{proof} The proof follows that of Lemma~\ref{l:bdM0} and is omitted.
\end{proof}
We have established the following result.
\begin{prop}\label{PP-TheoCor}
Fix  a set $\cG^\fil_{K,\ell}$ of admissible codimension two hypersurfaces,  and let $r_\fil$ be as in (\ref{e:SigmaPr}). Then there exists $\gamma=\gamma(r_\fil)>1$ such that for all
$\eps$ sufficiently small, the eigenvalues of $M$, defined in (\ref{PP-Def:M}) are given, to $O(\eps^\gamma)$, where
$$\gamma=1 + \frac{3r_\fil}{8}>1,$$ 
by its diagonal elements as indicated in (\ref{e:pr-M00-asymp}). In particular the pearling eigenvalues are given to leading
order by the diagonal elements of $M^{0,0}$ given in (\ref{e:pr-M00-asymp}).
\end{prop}
\begin{proof}
 Lemmas \ref{l:prM0} and \ref{l-MCor}, in conjunction with the estimates in Proposition\,\ref{p:M0pr} establish an $O(\eps^\gamma)$ bound on the operator
norm of the off-diagonal components of $M$.
\end{proof}

\subsection{Bounds on the off-diagonal operators}\label{PP-sec:BoundOffDiag}
The off-diagonal operator~$B=\Pi\LL_\fil\tilde\Pi$, is defined in~(\ref{def-BreakL-fil}). Recalling the decomposition (\ref{e:PP-mbbL}) of~$\LL_\fil$ into its dominant part~$\cL_\fil^2$ and asymptotically small, relatively bounded perturbations, our first step is to bound the dominant part,~$B_0=\Pi\cL_\fil^2\tilde\Pi$, of~$B$. Since, $B_0$, and $B_0^T$ enjoy the same bounds, for  simplicity we address the latter.

\begin{corol}\label{PP-equivnorm}
Fix  a set $\cG^\fil_{K,\ell}$ of admissible interfaces, then for each $m\in \NN_+$ there exists $C>0$ such that for each $\Gamma_\fil\in\cG^\fil_{K,\ell}$ and all $\vv\in\RR^{N_\fil}$ we have
\beq \label{EN-z-Pr}
\|\Delta_z^{m} v\|_{L^2(\Omega)} + \| \eps^{2m}\partial_s^m v\|_{L^2(\Omega)} \leq C \|\vv\|_{l^2},
\eeq
where $v:=\sum_{i\in\Sigma_{\fil,0}}\Theta_iv_i\spsi_{\fil,0} \in X_\fil(\Gamma_\fil)$
\end{corol}
\begin{proof} The proof is similar to that of \ref{PB-equivnorm}.
The bound on the first term on the left-hand side of  (\ref{EN-z-Pr}) follows from Lemma \ref{l:prM0} with the choice of $f$ given by
 $f(\vkappa) = (\Delta_z^{m}\spsi_{\fil,j},\Delta_z^{m}\spsi_{\fil,k})_\tJ$.
For the second term we observe from the form of $v$  that
\beq
 \eps^{2m}\partial_s^m v = \sum_{i\in\Sigma_{\fil}} \eps^{2m}\beta_i^m \spsi_{\fil,I(j)} \Theta_i v_i,
 \eeq
 and hence $\eps^{2m}\partial_s^m v$ is localized on $\Gamma_\fil$ and indeed lies in $X_\fil$.
 From the form, (\ref{e:SigmaPr}), of $\Sigma_{\fil,0}$ and~$\Sigma_{\fil,1}$ it follows that $\eps^{2m}\beta_i^m=\lambda_{\fil,I(j)}^m+O(\eps^{r_{\fil}/2})$.
 Applying  Lemma\,\ref{l:prM0} with the choice of $f$ given by
 $f(\vkappa) = \lambda_{\fil,I(j)}^m (\spsi_{\fil,j},\spsi_{\fil,k})_\tJ$, and using the orthonormality of the line diffusion's eigenmodes in the $L^2(\Gamma_\fil)$ norm yields the stated result.
\end{proof}

We use this to establish the following bound.

\begin{prop}\label{PP-prop:bound}
Fix $\Gamma_\fil\in\cG_{k,\ell}^\fil$, then for $\eps>0$ sufficiently small there exists a constant $C>0$ such that
\beq
\label{e:tPiLL-bnd-pr}
\| \tilde\Pi\LL_\fil v\|_{L^2(\Omega)} \leq C \eps \|\vv\|_{l^2},
\eeq
for all $\vv\in\RR^{N_{\fil,0}}$ with corresponding $v\in X_\fil$ given by (\ref{e:Xfil-expand}).
\end{prop}
\begin{proof}

Any $v \in X_\fil$ has the form (\ref{e:Xfil-expand}), which, using the Taylor expansion of~$J_\fil$, admits the expansion
\beq
v = \eps^{-1}\left( \ov_0(s)\hpsi_{\fil,0} - \ov_1(s)\hpsi_{\fil,1}\right)\left(1 + \frac12 \eps z\cdot\vkappa +  O(\eps^2)\right).
\eeq
Since $\eps^{-1}\cL_\fil^2(\ov_0\hpsi_{\fil,0} + \ov_1\hpsi_{\fil,1})\in X_\fil$, it lies in the kernel of $\tilde\Pi$, and only higher
order terms remain. Consequently we may write
\beq \label{e:Pi0cLv-pr}
\tilde\Pi_0\cL_b^2 v = \left(Q_0\ov+Q_1\ov_1\right),
\eeq
where the differential operators $Q_j$ admit an expansion of the form
 \beq \label{e:Pj-form-PR}
 Q_j:= Q_{j,0} + \eps Q_{j,1}+\eps^2 Q_{j,2} +\eps^3 Q_{j,3},
 \eeq
and each $Q_{j,k}$ is a $k$'th order differential operator in $\partial_s$ with coefficients
that are smooth and decay exponentially in $z$ and have uniform $L^\infty(\Gamma_\fil)$  bounds in $s$ that are independent of
$\eps>0$ sufficiently small. Taking the $L^2(\Omega)$ norm of (\ref{e:Pi0cLv-pr}), we transform to the local variables and integrate out
the $z$ dependence, the result is an expression of the form
\beq
 \|\tilde\Pi\cL_b^2 v \|_{L^2(\Omega)}^2 =  \eps^2 \int_{\Gamma_\fil}\left( \oQ_0\ov_0+\oQ_1\ov_1\right)^2\, ds,
 \eeq
 where the $\oQ_j$ are third order differential operators of the form (\ref{e:Pj-form-PR}) with $L^\infty(\Gamma_\fil)$ coefficients that only depend
 upon $s$.
In particular we deduce that
\beq
\| \eps^{k} \oQ_{j,k}\ov_j \|_{L^2(\Gamma_\fil)} \leq \|Q_{j,k}(1-\partial_s)^{-k/2}\|_{L_*^2(\Omega)} \|\eps^{k} (1-\partial_s)^{k/2} \ov \|_{L^2(\Gamma_\fil)},
\eeq
where the norm $\|\cdot\|_{L^2_*(\Omega)}$ denotes the induced $L^2\mapsto L^2$ operator norm on $L^2(\Gamma_\fil)$. From
classic elliptic regularity theory, the operator norm may be bounded, independent of $\eps>0$ sufficiently small,
while Lemma \ref{l-MCor} implies the existence of $C>0$ such that
\beq
\eps^{k}\|(1-\partial_s)^{k/2}\ov\|_{L^2(\Gamma_\fil)} \leq C \|\vv\|_{l^2},
\eeq
for $k=0,1, 2, 3.$
In particular we deduce the existence of $C>0$, chosen independent of $\eps>0$ sufficiently small that
\beq
\label{e:cL2-bnd-pr}
 \|\tilde{\Pi} \cL^2_b v \|_{L^2(\Omega)} \leq C \eps \|\vv\|_{l^2}.
\eeq

The extension of this bound to the full operator follows the steps at the end of the proof of Proposition~\ref{PB-prop:bound}. %
The result, (\ref{e:tPiLL-bnd-pr}) follows.
\end{proof}

\subsection{Localization of the geometric and pearling eigenvalues of $\LL_\fil$}\label{fil-sec:Spectrum}

The estimates from the previous section permit the localization of the small eigenvalues of $\LL_\fil$ as perturbations of
pearling and geometric eigenvalues of $M$, the estimates are similar to those in section \ref{PB-sec:Spectrum}, with the
coercivity afforded by Assumption\,\ref{CS-Assum:coercive}. This leads us to the following theorem,

\begin{theo}\label{PP-PearlingCond}
For space dimension $d=3$, fix a class $\cG_{K,\ell}^\fil$ of admissible filaments, and let $\LL_\fil$ be the second variation of
$\cF$ about a  filament morphology $u_\fil(\cdot; \Gamma_\fil, \hlam)$ given by (\ref{e:PR}). Fix $r=r_f^*:=\frac{8}{11}$ in the
definition (\ref{e:SigmaPr}) of $\Sigma_{\fil,0}$ and $\Sigma_{\fil,1}$.  Then there exists $C>0$ such that for all $\eps>0$ sufficiently small,
the set $\sigma(\LL_\fil)\cap (-\infty, C\eps^{r_\fil^*})$ consists of the union of pearling eigenvalues
$ \bigl\{\Lambda_{\fil;0,k}\,\bigl |\, k\in\Sigma_{\fil,0}(r_\fil^*)\bigr\}$
and the geometric eigenvalues $\bigl\{\Lambda_{\fil;1,k}\,\bigl | \, k\in\Sigma_{\fil,1}(r_\fil^*)\bigr\}$,
which satisfy the asymptotic expansions
\begin{align}
\label{e:Pevs-fil}
	\Lambda_{\fil;0,k} &= \frac{\eps}{\norm{\hpsi_{\fil,0}}_{L_R}^2}\left(P_{f,k,0}^2-\hlam S_\fil-\eta_d{\left(\norm{\hpsi_{\fil,0}'}_{L_R}^2+\lambda_{\fil,0}\norm{\hpsi_{\fil,0}}_{L_R}^2\right)}\right)+O\!\left(\eps^{2-r_\fil^*}\right), & {\rm for\,} k\in\Sigma_{\fil,0}(r_\fil^*),\\
\label{e:Gevs-fil}
		\Lambda_{\fil;1,k} &= \frac{\eps}{\norm{\hpsi_{\fil,0}}_{L_R}^2} P_{\fil,k,1}^2 + O\!\left(\eps^{2-r_\fil^*}\right), & {\rm for}\, k\in\Sigma_{\fil,1}(r_\fil^*),
\end{align}
where all quantities are as defined in Proposition \ref{p:M0pr}.  Moreover,
the associated codimension-two filament morphology is stable with respect to the pearling eigenvalues, if
\beq\label{fil-eq:PearlingStableCond}
    \hlam S_\fil +\eta_d{\left(\norm{\hpsi_{\fil,0}'}_{L_R}^2+\lambda_{\fil,0}\norm{\hpsi_{\fil,0}}_{L_R}^2\right)}<0,
\eeq
and is pearling unstable if this quantity is positive.
\end{theo}

\section{Conclusion}

Within the strong functionalization scaling of the FCH free energy, we have analyzed the stability of the bilayer and filament morphologies obtained by
the dressing procedure described in section \ref{sec:coordinates} applied to admissible codimension one interfaces and codimension two hypersurfaces.
The dressing procedure allows the arbitrary assignment of the spatially constant far-field chemical potential. Modulo the filament Assumptions\,
\ref{CS-Assum:Kernel} and \ref{CS-Assum:coercive}, we established that the pearling stability of these morphologies is independent of the
choice of the admissible codimensional one interface or codimension two hypersurface. Indeed, at leading order the pearling spectrum is
independent of the geometry hypersurface, and can be characterized in terms of the far-field chemical potential, the system parameters
$\eta_1$ and $\eta_2$, and the shape of double well potential $W$.

Admissible surfaces need not be simply connected, the results apply if the underlying hypersurfaces have disjoint components, so long as they are sufficiently smooth
and the reaches of each component do not overlap. The compact nature of the construction of the codimension-one and codimension-two morphologies, suggests that they
can be additively combined to generate hybrid morphologies with a common far-field chemical potential.
More specifically, given a codimension-one interface $\Gamma_b\in\cG_{K,\ell}^b$ and a
codimension-two hypersurface $\Gamma_\fil\in\cG_{K,\ell}^\fil$ such that the intersection of the reaches, $\Gamma_b^{\ell}\cap \Gamma_\fil^{\ell}$, is empty, then we may
form the hybred morphology
\beq\label{def-upf}
u_{b,\fil}(x;\Gamma_b,\Gamma_\fil,\hlam) = u_b(x;\Gamma_b,\hlam)+u_\fil(x;\Gamma_\fil,\hlam) - \eps\Frac{\hlam}{\alpha_-^2}.
\eeq
For each $K,\ell$ these hybrid functions reside in $H^4(\Omega)$ and are quasi-equilibria solutions of (\ref{e:FCH-static}). The spectral characterizations of Theorems \ref{Thm:Main-bl} and \ref{Thm:Main-pr} may be applied independently to the codimension-one and codimension-two morphologies.

The extension of these results to the weak scaling of the FCH free energy is not immediate. The complexity for the strong scaling arises through the large number of asymptotically small pearling eigenvalues, which requires uniform bounds on their interaction mediated through the non-constant hypersurface curvatures. For the weak scaling the number of pearling eigenvalues is $O(1)$ and may indeed be empty. However the difficulty that arises is that the pearling eigenvalues are smaller,
generically scaling as $O(\eps^2)$, and the coupling with the interfacial curvatures occurs at leading order, rather than at second order as in the strong scaling.
For the weak scaling the large curvatures could induce, or inhibit, the onset of the pearling bifurcation, or the pearling may be localized to regions of high curvature.
The resolution of these issue requires additional investigation.

\newpage

\bibliographystyle{apalike}
\bibliography{NKbib}

\begin{thebibliography}{}

\bibitem[Andreussi et~al., 2012]{andreussi2012revised}
Andreussi, O., Dabo, I., and Marzari, N. (2012).
\newblock Revised self-consistent continuum solvation in electronic-structure
  calculations.
\newblock {\em The Journal of Chemical Physics}, 136(6):064102.

\bibitem[Cahn and Hilliard, 1958]{Cahn1958free}
Cahn, J.~W. and Hilliard, J.~E. (1958).
\newblock Free energy of a nonuniform system. i. interfacial free energy.
\newblock {\em The Journal of Chemical Physics}, 28(2):258--267.

\bibitem[Chen, 1994]{chen1994spectrum}
Chen, X. (1994).
\newblock Spectrum for the allen-chan, chan-hillard, and phase-field equations
  for generic interfaces.
\newblock {\em Communications in Partial Differential Equations},
  19(7-8):1371--1395.

\bibitem[Christlieb et~al., tted]{Christlieb2017Competition}
Christlieb, A., Kraitzman, N., and Promislow, K. (submitted).
\newblock Competition and complexity in amphiphilic polymer morphology.
\newblock {\em submitted}.

\bibitem[Dai and Promislow, 2013]{dai2013geometric}
Dai, S. and Promislow, K. (2013).
\newblock Geometric evolution of bilayers under the functionalized
  cahn--hilliard equation.
\newblock In {\em Proceedings of the Royal Society of London A: Mathematical,
  Physical and Engineering Sciences}, volume 469, page 20120505. The Royal
  Society.

\bibitem[Dai and Promislow, 2015]{DAI2015COMPETITIVE}
Dai, S. and Promislow, K. (2015).
\newblock Competitive geometric evolution of amphiphilic interfaces.
\newblock {\em SIAM Journal on Mathematical Analysis}, 47(1):347--380.

\bibitem[Doelman et~al., 2014]{Doelman2014meander}
Doelman, A., Hayrapetyan, G., Promislow, K., and Wetton, B. (2014).
\newblock Meander and pearling of single-curvature bilayer interfaces in the
  functionalized cahn--hilliard equation.
\newblock {\em SIAM Journal on Mathematical Analysis}, 46(6):3640--3677.

\bibitem[Gompper and Schick, 1990]{gompper1990correlation}
Gompper, G. and Schick, M. (1990).
\newblock Correlation between structural and interfacial properties of
  amphiphilic systems.
\newblock {\em Physical Review Letters}, 65(9):1116.

\bibitem[Hayrapetyan and Promislow, 2014]{HAYRAPETYAN2014SPECTRA}
Hayrapetyan, G. and Promislow, K. (2014).
\newblock Spectra of functionalized operators arising from hypersurfaces.
\newblock {\em Zeitschrift f{\"u}r Angewandte Mathematik und Physik}, pages
  1--32.

\bibitem[Hayrapetyan and Promislow, 2016]{HAYRAPETYAN2016Nonlinear}
Hayrapetyan, G. and Promislow, K. (2016).
\newblock Nonlinear stability of radial bilayers under the functionalized
  cahn-hilliard gradient flow.
\newblock {\em Preprint}.

\bibitem[Promislow and Wu, 2015]{PW-15}
Promislow, K. and Wu, Q. (2015).
\newblock Existence of pearled patterns in the planar functionalized
  cahn-hilliard equation.
\newblock {\em J. Differential Equations}, 259:3298--3343.

\bibitem[Scherlis et~al., 2006]{scherlis2006unified}
Scherlis, D.~A., Fattebert, J.-L., Gygi, F., Cococcioni, M., and Marzari, N.
  (2006).
\newblock A unified electrostatic and cavitation model for first-principles
  molecular dynamics in solution.
\newblock {\em The Journal of Chemical Physics}, 124(7):074103.

\bibitem[Teubner and Strey, 1987]{teubner1987origin}
Teubner, M. and Strey, R. (1987).
\newblock Origin of the scattering peak in microemulsions.
\newblock {\em The Journal of Chemical Physics}, 87(5):3195--3200.

\bibitem[Zhu et~al., 2009]{zhu2009tuning}
Zhu, J., Ferrer, N., and Hayward, R.~C. (2009).
\newblock Tuning the assembly of amphiphilic block copolymers through
  instabilities of solvent/water interfaces in the presence of aqueous
  surfactants.
\newblock {\em Soft Matter}, 5(12):2471--2478.

\bibitem[Zhu and Hayward, 2012]{zhu2012interfacial}
Zhu, J. and Hayward, R.~C. (2012).
\newblock Interfacial tension of evaporating emulsion droplets containing
  amphiphilic block copolymers: Effects of solvent and polymer composition.
\newblock {\em Journal of Colloid and Interface Science}, 365(1):275--279.

\end{thebibliography}
\end{document}